\documentclass[11 pt]{amsart}

\usepackage{amsmath}
\usepackage{latexsym,amsfonts,amssymb,mathrsfs, mathtools}
\usepackage{mathdots}
\usepackage{color}
\usepackage [all,2cell,color]{xy}
\usepackage{enumitem}

\usepackage{dsfont} 
\usepackage{pigpen}

\usepackage{tikz-cd}

\UseAllTwocells

\usepackage{expl3} 
\ExplSyntaxOn  
\def\minusone#1
{\int_eval:n {#1 - 1}}
\ExplSyntaxOff
\ExplSyntaxOn   
\def\sumplusone#1#2
{\int_eval:n {#1 +#2+1}}
\ExplSyntaxOff

\usepackage{hyperref}
\usepackage[capitalise]{cleveref}

\makeatletter
\def\paragraph{\@startsection{paragraph}{4}%
  \z@\z@{-\fontdimen2\font}
  {\normalfont\bfseries}}
\makeatother

\newcommand{\cA}{\mathcal{A}}
\newcommand{\cB}{\mathcal{B}}
\newcommand{\cC}{\mathcal{C}}
\newcommand{\cD}{\mathcal{D}}
\newcommand{\cE}{\mathcal{E}}
\newcommand{\cF}{\mathcal{F}}
\newcommand{\cG}{\mathcal{G}}

\newcommand{\cM}{\mathcal{M}}

\newcommand{\cQ}{\mathcal{Q}}

\newcommand{\cS}{\mathcal{S}}

\newcommand{\cW}{\mathcal{W}}

\newcommand{\cat}{\cC\!\mathit{at}}

\newcommand{\set}{\cS\!\mathit{et}}

\newcommand{\sset}{\mathit{s}\set}
\newcommand{\gpd}{\mathcal Gpd}

\newcommand{\sdot}{S_{\bullet}}
\newcommand{\sdotn}[1]{S_{#1}}
\newcommand{\sdote}{\sdot^e}
\newcommand{\sdoten}[1]{S_{#1}^{e}}
\newcommand{\sdotw}{\sdot^{\mathrm{Wald}}}
\newcommand{\sdotwn}[1]{\sdotn{#1}^{\mathrm{Wald}}}
\newcommand{\sdots}{\sdot^s}
\newcommand{\sdotsn}[1]{S_{#1}^{s}}
\newcommand{\sdotpe}{\sdot^{pe}}
\newcommand{\rel}{\operatorname{rel}}
\newcommand{\sdotrel}{\sdot^{\rel}}
\newcommand{\sdotreln}[1]{S_{#1}^{\rel}}
\newcommand{\nrel}{N^{\rel}}

\newcommand{\pcat}{\mathcal{P}}
\newcommand{\epi}{\twoheadrightarrow}
\newcommand{\mono}{\rightarrowtail}

\newcommand{\excat}{\mathcal{E}x\mathcal{C}at} 
\newcommand{\pexcat}{\mathcal{PE}x\mathcal{C}at} 
\newcommand{\stqcat}{\mathcal{S}t\mathcal{QC}at}

\newcommand{\targetcat}{\cC}
\newcommand{\exactcat}{\cA}
\newcommand{\sS}{\targetcat^{\Delta^{\operatorname{op}}}}

\newcommand{\sa}[1]{{#1}^{\Sigma^{\operatorname{op}}}}
\newcommand{\sagpd}{\sa{\gpd}}
\newcommand{\sasset}{\sa{\sset}} 
 
\newcommand{\saS}{\sa{\targetcat}}

\newcommand{\aug}{A}

\DeclareMathOperator{\id}{id}

\DeclareMathOperator{\Fun}{Fun}

\DeclareMathOperator{\Map}{Map}

\DeclareFontFamily{OT1}{pzc}{}
\DeclareFontShape{OT1}{pzc}{m}{it}{<-> s * [1.10] pzcmi7t}{}
\DeclareMathAlphabet{\mathpzc}{OT1}{pzc}{m}{it}

\DeclareMathOperator{\op}{op}

\DeclareMathOperator{\pr}{pr}

\DeclareMathOperator{\Ar}{Ar}

\newcommand{\wW}[1]{\operatorname{W}[#1]}

\newcommand{\htimes}[1]{\underset{#1}{\overset{h}{\times}}}
\newcommand{\ttimes}[1]{\underset{#1}{\times}}

\newcommand{\twotimes}[1]{\underset{#1}{\overset{2}{\times}}}

\newcommand{\inda}{q} 
\newcommand{\indb}{r}
\newcommand{\indc}{k}
\newcommand{\indd}{\ell}

\newcommand{\sourceverd}{s^v}
\newcommand{\sourcehord}{s^h}
\newcommand{\targetverd}{t^v}
\newcommand{\targethord}{t^h}

\newcommand{\sourcever}{s_v}
\newcommand{\sourcehor}{s_h}
\newcommand{\targetver}{t_v}
\newcommand{\targethor}{t_h}

\usetikzlibrary{positioning}
\usetikzlibrary{decorations.pathreplacing}
\usetikzlibrary{arrows, decorations.markings} 

\pgfarrowsdeclarecombine{twotriang}{twotriang}{stealth}{stealth}%
{stealth}{stealth}
\tikzset{arrow/.style={-stealth}}
\tikzset{arrowshorter/.style={-stealth, shorten <=2pt, shorten >=2pt}}
\tikzset{arrowmuchshorter/.style={-stealth, shorten <=7pt, shorten >=6pt}}
\tikzset{mono/.style={>-stealth}} 
\tikzset{epi/.style={-twotriang}} 
\tikzset{twoarrowlonger/.style={double,double distance=1.5pt,
shorten <=5pt,shorten >=6pt,
decoration={markings,mark=at position -4pt with {\arrow[scale=1.75]{>}}},
preaction={decorate}}} 
\tikzset{twoarrowlongerred/.style={thick, red,double,double distance=1.5pt,
shorten <=5pt,shorten >=6pt,
decoration={markings,mark=at position -4pt with {\arrow[scale=1.75, red]{>}}},
preaction={decorate}}} 

\usetikzlibrary{backgrounds,shapes,arrows,calc} 

\pgfdeclarelayer{background}
\pgfsetlayers{background,main}

\tikzset{mapstikz/.style={-stealth, 
decoration={markings,mark=at position 0pt with {\arrow[scale=0.5]{|}}}, preaction={decorate}}}

\tikzset{mapstikz2/.style={-stealth, 
decoration={markings,mark=at position 0pt with {\arrow[scale=1]{|}}}, preaction={decorate}}}

\tikzset{dot/.style={circle,draw,fill,inner sep=1pt},
    arrowinline/.style={-stealth,thick,shorten <=2pt,shorten >=2pt}}

\tikzstyle{category1} = [rectangle, rounded corners, minimum width=2.2cm, minimum height=1.5cm,text centered, draw=black, fill=green!30, text width=2.3cm]
\tikzstyle{categoryDisc} = [rectangle, rounded corners, minimum width=2.1cm, minimum height=1cm,text centered, draw=black, fill=yellow!30, text width=2.1cm]
\tikzstyle{category2} = [rectangle, rounded corners, minimum width=2.1cm, minimum height=1cm,text centered, draw=black, fill=blue!30, text width=2.1cm]
\tikzstyle{category3} = [rectangle, rounded corners, minimum width=2.8cm, minimum height=1cm,align=center, draw=black, fill=blue!30]

\tikzstyle{higharrow}=[thick,red]

\tikzset{highnode/.style={red}}

\theoremstyle{plain}   
\newtheorem{thm}{Theorem}[section]
\makeatletter\let\c@thm\c@thm\makeatother
\newtheorem{cor}{Corollary}[section]
\makeatletter\let\c@cor\c@thm\makeatother
\newtheorem{lem}{Lemma}[section]
\makeatletter\let\c@lem\c@thm\makeatother
\newtheorem{prop}{Proposition}[section]
\makeatletter\let\c@prop\c@thm\makeatother

\makeatletter\let\c@claim\c@thm\makeatother

\newtheorem*{unnumberedtheoremA}{Theorem A}
\newtheorem*{unnumberedtheoremB}{Theorem B}

\theoremstyle{definition}

\newtheorem{defn}{Definition}[section]
\makeatletter\let\c@defn\c@thm\makeatother
\newtheorem{const}{Construction}[section]
\makeatletter\let\c@const\c@thm\makeatother
\newtheorem{notn}{Notation}[section]
\makeatletter\let\c@notn\c@thm\makeatother

\makeatletter\let\c@assumption\c@thm\makeatother

\theoremstyle{remark}

\newtheorem{rmk}{Remark}[section]
\makeatletter\let\c@rmk\c@thm\makeatother
\newtheorem{ex}{Example}[section]
\makeatletter\let\c@ex\c@thm\makeatother

\makeatletter\let\c@observation\c@thm\makeatother

\makeatletter\let\c@digression\c@thm\makeatother

\makeatletter
\let\c@equation\c@thm
\numberwithin{equation}{section}
\makeatother

\newcommand{\newrefformat}[2]{}

\crefname{lem}{Lemma}{Lemmas}
\crefname{thm}{Theorem}{Theorems}
\crefname{defn}{Definition}{Definitions}
\crefname{notn}{Notation}{Notations}
\crefname{const}{Construction}{Constructions}
\crefname{prop}{Proposition}{Propositions}
\crefname{rmk}{Remark}{Remarks}
\crefname{cor}{Corollary}{Corollaries}
\crefname{equation}{Diagram}{Diagram} 
\crefname{ex}{Example}{Examples}
\crefname{section}{Section}{Sections}
\crefname{subsection}{Section}{Sections} 
\crefname{digression}{Digression}{Digressions}
\crefname{assumption}{Assumption}{Assumption}

\begin{document}
\title{Comparison of Waldhausen constructions}

\author{Julia E. Bergner}
\address{Department of Mathematics, University of Virginia, Charlottesville, VA 22904, USA}
\email{jeb2md@virginia.edu}

\author{Ang\'{e}lica M. Osorno}
\address{Department of Mathematics, Reed College, Portland, OR 97202, USA}
\email{aosorno@reed.edu}

\author{Viktoriya Ozornova}
\address{Fakult\"at f\"ur Mathematik, Ruhr-Universit\"at Bochum, D-44780 Bochum, Germany}
\email{viktoriya.ozornova@rub.de}

\author{Martina Rovelli}
\address{Department of Mathematics,
Johns Hopkins University,
Baltimore, MD 21218, USA}
\email{mrovelli@math.jhu.edu}

\author{Claudia I. Scheimbauer}
\address{Department of Mathematical Sciences, Norwegian University of Science and Technology (NTNU), 7491 Trondheim, Norway}
\email{scheimbauer@tum.de}

\date{\today}



\dedicatory{Dedicated to the memory of Thomas Poguntke}

\thanks{The first-named author was partially supported by NSF CAREER award DMS-1659931. The second-named author was partially supported by a grant from the Simons Foundation (\#359449, Ang\'elica Osorno), the Woodrow Wilson Career Enhancement Fellowship and NSF grant DMS-1709302. The fourth-named and fifth-named authors were partially funded by the Swiss National Science Foundation, grants P2ELP2\textunderscore172086 and P300P2\textunderscore 164652. The fifth-named author was partially funded by the Bergen Research Foundation and would like to thank the University of Oxford for its hospitality. The authors would like to thank the Isaac Newton Institute for Mathematical Sciences, Cambridge, for support and hospitality during the program ``Homotopy Harnessing Higher Structures" where work on this paper was undertaken. This work was supported by EPSRC grant no EP/K032208/1.} 

\begin{abstract}
In previous work, we develop a generalized Waldhausen $\sdot$-construction whose input is an augmented stable double Segal space and whose output is a  $2$-Segal space.  Here, we prove that this construction recovers the previously known $\sdot$-constructions for exact categories and for stable and exact $(\infty,1)$-categories, as well as the relative $\sdot$-construction for exact functors.
\end{abstract}

\maketitle

\setcounter{tocdepth}{1}
\tableofcontents

\section*{Introduction}

Waldhausen's approach to understanding higher $K$-theory [Wal85] allowed for more general inputs than previous approaches, of particular interest that of retractive spaces. His formulation is categorical in nature and has been reinterpreted and generalized in several different ways for different inputs in the literature, for example in \cite{barwickKtheory}, \cite{DK}, and \cite{BGT}.

A unifying property of these variants, all called {\em $\sdot$-constructions}, is that their outputs are simplicial spaces which can then be used to construct algebraic $K$-theory.  More recently, in many cases of interest, they have been shown in \cite{DK} and \cite{GalvezKockTonks} to have the additional structure of a {\em $2$-Segal} or {\em decomposition} space.

In our paper \cite{BOORS2}, we develop a generalization
of Waldhausen's $\sdot$-construction whose input is given by an {\em augmented stable double Segal space}.  The structure of this input is meant to capture the essential features of exact categories and generalizations thereof to which one often applies known $\sdot$-constructions.  Furthermore, the output of our construction is still a $2$-Segal space.

Here, we show that our $\sdot$-construction is indeed a generalization, in that it recovers previously developed ones, namely those for exact categories, stable $(\infty,1)$-categories, proto-exact $(\infty,1)$-categories, and Waldhausen's relative $\sdot$-construction for exact functors.  More specifically, we construct nerve functors
for each of these kinds of inputs and we show in \cref{thm nerve 1,thm nerve 2,thm nerve 3,prop relative asds} that the respective outputs are augmented stable double Segal spaces.

\begin{unnumberedtheoremA}
The nerves of exact categories, stable $(\infty,1)$-categories, proto-exact $(\infty,1)$-categories, and exact functors are augmented stable double Segal spaces.
\end{unnumberedtheoremA}

We go on to show in \cref{comparisonsdotexact,comparisonsdotstable,comparisonsdotexactinfinity,thm comparison relative} that the generalized $\sdot$-construction from \cite{BOORS2} agrees with the previously existing ones from \cite{waldhausen}, \cite{barwickKtheory}, \cite{BGT}, and \cite{DK}. 

\begin{unnumberedtheoremB}
The generalized $\sdot$-construction composed with the appropriate nerve functor is equivalent to the previous $\sdot$-constructions in each of these contexts.
\end{unnumberedtheoremB}

Furthermore, these constructions recover those from our paper \cite{BOORS}, in which we treated a discrete version of the generalized $\sdot$-construction. 

We summarize all variants in the following diagram, denoting the different variants of nerves by $N^{e}, N^{s}, N^{pe}$, and $N^{rel}$.
{\small
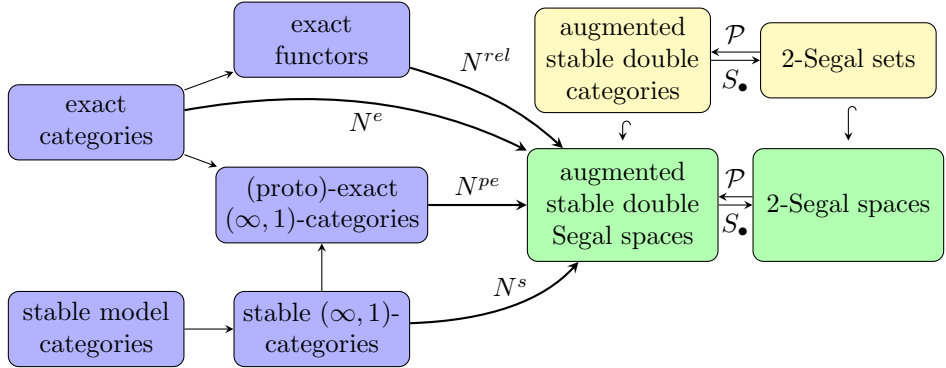
\begin{center}
 \begin{tikzpicture}[yscale=0.55]
\begin{scope}[xshift=-1cm]
  \draw (6,-2) node[category2] (stableMC){stable model categories};
  \draw (9,-2) node[category2] (stableIC){stable $(\infty,1)$-categories};
   \draw (6,3) node[category2] (exact){exact \\ categories};
   \draw (9,5) node[category2] (exactf){exact \\functors};
   \draw (9,1) node[category3] (protoexactIC){(proto)-exact \\ $(\infty,1)$-categories};
   
\end{scope}
   \draw (12,1) node[category1] (Sigma){augmented stable double Segal spaces};
   \draw (12,4.5) node[categoryDisc] (ASDC){augmented stable double categories};
   \draw (15,1) node[category1] (U2S){$2$-Segal spaces};
   \draw (15, 4.5) node[categoryDisc] (U2d){$2$-Segal sets}; 
   
   \draw[-stealth] (stableMC)--(stableIC);
   \draw[-stealth] (stableIC)--(protoexactIC);
   \draw[-stealth] (exact)--(protoexactIC);
    \draw[-stealth] (exact)--(exactf);
   
   \draw[right hook-stealth, shorten <=0.1cm, shorten >=0.1cm] (ASDC)--(Sigma);
   \draw[right hook-stealth, shorten <=0.1cm, shorten >=0.1cm] (U2d)--(U2S);
 
   \draw[-stealth, thick] (protoexactIC)-- node[above](a1){$N^{pe}$}(Sigma);
   \path[-stealth, thick] (stableIC) edge [bend right] node[above](a1){$N^{s}$}(Sigma);
   \path[-stealth, thick] (exact) edge [bend left] node[below](a1){$N^{e}$}(Sigma);
   \path[-stealth, thick] (exactf) edge [bend left = 15] node[above right, pos=0.2](a1){$N^{rel}$}(Sigma);
   
   \draw[stealth-] ([yshift=0.2cm]Sigma.east)--node[above](a2){$\pcat$} ([yshift=0.2cm]U2S.west);
  \draw[stealth-] (U2S)--node[below](a3){$\sdot$}(Sigma);

   \draw[stealth-] ([yshift=0.2cm]ASDC.east)--node[above](a2){$\pcat$} ([yshift=0.2cm]U2d.west);
  \draw[stealth-] (U2d)--node[below](a3){$\sdot$}(ASDC);
 \end{tikzpicture}
\end{center}
}

This work relates to a number of ideas currently being developed.  For example, some work on higher dimensional $\sdot$-constructions in the context of higher Segal spaces has been done by Poguntke \cite{Poguntke}, and it would be interesting to extend our approach to this context as well.  We also note that the relative $\sdot$-construction we consider here does not use the notion of relative $2$-Segal spaces of Walde \cite{Walde}, and Young \cite{Young}. Finally, a slight generalization of $N^{pe}$ should provide a nerve functor for Penney's augmented proto-exact $(\infty,1)$-categories \cite{PenneyHall}. 

For sake of exposition, we start with the classical, most familiar setting of exact categories. Then we move to the homotopical context of stable $(\infty,1)$-categories. Finally, we discuss the most general case of exact or proto-exact $(\infty,1)$-categories, the proofs for which are straightforward generalizations of the previous two cases.  In the last section, we change our input to exact functors between exact categories.

For the reader's convenience, we include an appendix with some of the results we use concerning quasi-categories.

\subsection*{Acknowledgements}

We would like to thank Toby Dyckerhoff for suggesting to look at the relative Waldhausen construction.

\section{Augmented stable double Segal objects}

Dyckerhoff and Kapranov define $2$-Segal objects using higher-dimensional analogues of the Segal maps used to define Segal spaces, in the sense that they arise from polygonal decompositions of polygons \cite{DK}. They originally focus on $2$-Segal spaces satisfying an additional unitality assumption, and these unital $2$-Segal spaces exactly agree with the decomposition spaces of G\'alvez-Carrillo, Kock, and Tonks \cite{GalvezKockTonks}.  However, a recent surprising result of Feller, Garner, Kock, Underhill-Proulx, and Weber \cite{FGKUW} shows that every $2$-Segal space is unital.  We thus frequently omit the adjective ``unital'' unless we find it useful for clarification.

In this paper, we consider $2$-Segal objects in groupoids and in simplicial sets; we refer to the latter as \emph{$2$-Segal spaces}.  Roughly speaking, a $2$-Segal space can be thought of as a structure similar to a category up to homotopy, but in which composition might be multiply defined, or not defined at all.  However, composition is associative 
whenever it is defined.  For a more detailed discussion of this interpretation, we refer the reader to \cite[\textsection 3.3]{DK}.  

We do not need the formal definitions here, since the comparisons that we make between $2$-Segal objects are given by weak equivalences of groupoids or simplicial sets at each simplicial level.  We refer the reader to \cite{DK} for a more rigorous treatment.

In our paper \cite{BOORS2}, we establish an equivalence between $2$-Segal objects and augmented stable double Segal objects, so we now turn to the definition of the latter, which is the structure of primary interest to us here.  We give the definition in stages, and at times in specialized variants that are suitable for our purposes here. 

Let $\targetcat$ be a monoidal combinatorial model category, again with primary examples the model categories for groupoids and simplicial sets.  We begin with double Segal objects, which model up-to-homotopy double categories.  Informally, a double category consists of objects, horizontal morphisms, vertical morphisms, and squares, satisfying some compatibility conditions.  (See \cite{Ehresmann}, \cite{FiorePaoliPronk} and \cite{GP} for details.)  The two kinds of morphisms are reflected below in the two simplicial directions. Details for the interpretation of the following definition can be found in \cite{HaugsengIteratedSpans}. 

\begin{defn} \label{doublesegal}
A bisimplicial object $Y$ in $\targetcat$ is a \emph{double Segal object} if for every $\inda,\indb \geq 1$ the maps
\[ Y_{\inda,\indb} \rightarrow \underbrace{Y_{\inda,1} \htimes{Y_{\inda,0}} \cdots \htimes{Y_{\inda,0}} Y_{\inda,1}}_{\indb} \text{ and }Y_{\inda,\indb} \rightarrow \underbrace{Y_{1,\indb} \htimes{Y_{0,\indb}} \cdots \htimes{Y_{0,\indb}} Y_{1,\indb}}_{\inda}  \]
are weak equivalences in $\targetcat$.  Here the left-hand map is induced by the inclusion of the spine into the standard $\indb$-simplex in the second variable, whereas the right-hand map is induced by the inclusion of the spine into the standard $\inda$-simplex in the first variable. 
\end{defn}

Continuing the analogy between double Segal objects and double categories, some of the bisimplicial structure will play the role of horizontal and vertical source and target maps, so we introduce the corresponding notation.  
\begin{notn} 
For ease of notation, we denote an object $([k], [\ell])$ in $\Delta \times \Delta$ simply by $(k,\ell)$.  We denote by $s^h, t^h\colon (k,0) \to (k,\ell)$ the maps given by identity in the first component and $0\mapsto 0$ and $0\mapsto \ell$ in the second component, respectively. We similarly denote by $s^v,t^v\colon (0,\ell) \to (k,\ell)$ the maps $(0\mapsto 0,\id)$ and $(0\mapsto k,\id)$, respectively.
\end{notn}

Next we consider stability of a double Segal object. The idea behind stability is the following. If a double Segal object is taken to model an up-to-homotopy double category, then stability requires a square to be determined, up to homotopy, by either its cospan or its span.  The idea is to encode the analogue of squares being simultaneously cartesian and cocartesian but in the context of a double category.

\begin{defn}\label{babystability} 
A double Segal object $Y$ in $\targetcat$ is \emph{stable} if the squares
\\
\begin{subequations}
\noindent\begin{minipage}{0.45\linewidth}
\begin{equation}\label{eq:babystabilityspan}
 \begin{tikzcd}
(0,0 )\arrow{r}{\sourcehord} \arrow[d, "\sourceverd", swap] & (0, 1) \arrow{d}{\sourceverd} \\
(1,0 )\arrow[r,"\sourcehord", swap] & (1,1)
\end{tikzcd}
\end{equation}
    \end{minipage}
    \begin{minipage}{0.1\columnwidth}\centering
    and 
    \end{minipage}
    \begin{minipage}{0.45\linewidth}
\begin{equation}\label{eq:babystabilitycospan}
 {\begin{tikzcd}
(0, 0) \arrow{r}{\targethord} \arrow[d,"\targetverd", swap] &( {0},1) \arrow{d}{\targetverd} \\
(1,0 )\arrow[r,"\targethord", swap] & (1,1)
\end{tikzcd}}
\end{equation}
    \end{minipage}
  \end{subequations}
\\induce weak equivalences
$$Y_{0,1} \htimes{Y_{0,0}} Y_{1,0} \overset{\simeq}{\longleftarrow} Y_{1,1} \overset{\simeq}{\longrightarrow} Y_{1,0} \htimes{Y_{0,0}} Y_{0,1}.$$
\end{defn}

\begin{rmk}\label{adultstability}
The above is not the original definition of stability given in \cite{BOORS2}. As we show therein, the two definitions are equivalent, and the one given here is better suited to our usage. The original definition asks for similarly defined maps
\[ Y_{0,r} \htimes{Y_{0,0}} Y_{q,0} \overset{\simeq}{\longleftarrow} Y_{q,r} \overset{\simeq}{\longrightarrow} Y_{q,0} \htimes{Y_{0,0}} Y_{0,r} \]
to be weak equivalences for all $q,r\geq 1$. See also \cite[Lemma 2.3.3]{Carlier}.
\end{rmk}

Finally, we consider the extra data of an augmentation, which is encoded by modifying the category $\Delta \times \Delta$ by adjoining a terminal object.

\begin{defn} \label{categorysigma} \label{definitionpreaugmented}
Let $\Sigma$ be the category obtained from $\Delta\times \Delta$ by adding a new terminal object, denoted by $[-1]$:
\[ \begin{tikzcd}
{[-1]}  \\
\mbox{} & 
(0,0)
\arrow{lu}{} 
 \arrow[r, arrow, shift left=1ex] \arrow[r, arrow, shift right=1ex] 
\arrow[d, arrow, shift left=1ex] \arrow[d, arrow, shift right=1ex]& 
(1,0)\arrow[l, arrow]  \arrow[r, arrow] \arrow[r, arrow, shift left=1.5ex] \arrow[r, arrow, shift right=1.5ex]
\arrow[d, arrow, shift left=1ex] \arrow[d, arrow, shift right=1ex]& 
(2,0) 
\arrow[l, arrowshorter, shift left=0.75ex] \arrow[l, arrowshorter, shift right=0.75ex] 
\arrow[d, arrow, shift left=1ex] \arrow[d, arrow, shift right=1ex]&[-0.8cm]\cdots\\
\mbox{} & 
(0,1) \arrow[r, arrow, shift left=1ex] \arrow[r, arrow, shift right=1ex]
\arrow[u, arrowshorter]
\arrow[d, arrow] \arrow[d, arrow, shift left=1.5ex] \arrow[d, arrow, shift right=1.5ex]
 & 
(1,1)  \arrow[l, arrow]  \arrow[r, arrow] \arrow[r, arrow, shift left=1.5ex] \arrow[r, arrow, shift right=1.5ex] 
\arrow[u, arrowshorter]
\arrow[d, arrow] \arrow[d, arrow, shift left=1.5ex] \arrow[d, arrow, shift right=1.5ex]& 
(2,1) \arrow[l, arrowshorter, shift left=0.75ex] \arrow[l, arrowshorter, shift right=0.75ex] 
\arrow[u, arrowshorter]
\arrow[d, arrow] \arrow[d, arrow, shift left=1.5ex] \arrow[d, arrow, shift right=1.5ex]&\cdots\\
\mbox{} & 
(0,2) \arrow[r, arrow, shift left=1ex] \arrow[r, arrow, shift right=1ex]
\arrow[u, arrowshorter, shift left=0.75ex] \arrow[u, arrowshorter, shift right=0.75ex] 
& 
(1,2)  \arrow[l, arrow]  \arrow[r, arrow] \arrow[r, arrow, shift left=1.5ex] \arrow[r, arrow, shift right=1.5ex]
\arrow[u, arrowshorter, shift left=0.75ex] \arrow[u, arrowshorter, shift right=0.75ex] & 
(2,2) \arrow[l, arrowshorter, shift left=0.75ex] \arrow[l, arrowshorter, shift right=0.75ex] 
\arrow[u, arrowshorter, shift left=0.75ex] \arrow[u, arrowshorter, shift right=0.75ex] &
\cdots\\[-0.7cm]
\mbox{} & \vdots & \vdots & \vdots &\ddots.\\
\end{tikzcd}
\]
A {\em preaugmented bisimplicial object} in $\targetcat$ is a functor $Y\colon\Sigma^{\op} \to\targetcat$, we denote the category of such functors by $\saS$.
\end{defn}

The canonical inclusion $i\colon \Delta\times \Delta \to \Sigma$ allows us to extend \cref{doublesegal,,babystability} for stable and double Segal objects from bisimplicial objects to preaugmented bisimplicial objects by requiring the condition in question on the underlying bisimplicial object; we refer the reader to \cite{BOORS2} for details.

We can now define augmented double Segal objects.  While there is a more general definition of augmented bisimplicial object, the following definition is equivalent in the presence of the double Segal condition. 

\begin{defn} \label{babyaugmentation}
A preaugmented double Segal object $Y$ in $\targetcat$ is \emph{augmented} if the composites
\begin{align}
 Y_{1,0} {^{\targetver}\htimes{Y_{0,0}}} Y_{-1}\xrightarrow{\pr_1} Y_{1,0} \xrightarrow{\sourcever} Y_{0,0} \label{eqn: augmented 1A}\\
Y_{0,1}{^{\sourcehor}\htimes{Y_{0,0}}}  Y_{-1}\xrightarrow{\pr_1} Y_{0,1} \xrightarrow{\targethor} Y_{0,0} \label{eqn: augmented 2A}
\end{align}
are weak equivalences in $\targetcat$.
\end{defn}

The ordinal sum $\Delta\times\Delta\to\Delta$ extends to a functor $p \colon \Sigma \to \Delta$ along the canonical inclusion $i$ satisfying
\[ p\left((\inda,\indb)\to[-1]\right):=\left([\inda+1+\indb]\to[0]\right).\]
In particular, on objects $p$ is given by
\[p(\inda, \indb):=[\inda+1+\indb]\quad\text{ and }\quad p(-1):=[0]. \]
The induced functor $p^*\colon \sS\to\saS$ admits a right adjoint 
$p_*\colon\saS\to\sS $ given by right Kan extension.  

We now introduce our generalized $\sdot$-construction, the central object of study in this paper. 

\begin{defn} \label{sdot}
The \emph{generalized $\sdot$-construction} is the functor $\sdot=p_*$, the right adjoint to $p^*$.
\end{defn}

The main result of our previous paper is that this adjunction is indeed very strong, in the following sense.

\begin{thm}[\cite{BOORS2}] \label{QuillenEquivalence}
The adjoint pair $(p^*, \sdot)$  induces a Quillen equivalence between a model structure for $2$-Segal objects and a model structure for augmented stable double Segal objects.
\end{thm}

The following example will be key in later arguments.

\begin{ex}[\cite{BOORS2}]
\label{nerveofWn}
\label{rmkpstar(Delta)=W}
Consider the preaugmented bisimplicial set
\[ \wW{n}\coloneqq p^*\Delta[n]. \]

A careful verification shows that, for any $n\ge0$, the preaugmented bisimplicial set $\wW{n}$ is given by
\[ \wW{n}_{\indc, \indd}=\left\{(i_0,\dots,i_\indc,j_0,\dots,j_\indd) \mid 0\leq i_0\leq\dots\leq i_{\indc}\leq j_0\leq\dots\leq j_{\indd}\leq n\right\} \]
for any $\indc, \indd \geq 0$ and by
\[ \wW{n}_{-1}=\left\{(i,i) \mid 0\leq i\leq n\right\}. \]
Degeneracy maps are given by repeating the appropriate index, face maps by removing according to that index, and the augmentation map by the canonical inclusion.  In particular, $\wW{0}$ is the representable functor associated to the object $[-1]$.

Informally, we can think of $\wW{n}$ as (the nerve of) a double category pictured below for $n=4$:
\begin{equation}
\label{PictureW4}
\begin{tikzpicture}[scale=0.7, inner sep=0pt, font=\footnotesize, baseline=(current  bounding  box.center)]
\begin{scope}
     \draw (1,0) node(a00){00};
\draw (2,0) node(a01){{\footnotesize  01 }};
\draw (2,-1) node(a11) {{\footnotesize  11 }};
\draw  (3, -1)   node(a12){{\footnotesize  12 }};
\draw  (3, 0)   node(a02){{\footnotesize  02 }};
\draw (3,-2) node (a22){{\footnotesize  22 }};
\draw  (4, 0)   node (a03){{\footnotesize  03 }};
\draw  (4, -1)   node (a13){{\footnotesize  13 }};
\draw (4, -2)   node (a23){{\footnotesize  23 }};
\draw (4, -3) node (a33){{\footnotesize  33 }};
\draw (5, 0)   node (a04){{\footnotesize  04 }};
\draw  (5, -1)  node (a14){{\footnotesize  14 }};
\draw  (5, -2)  node (a24){{\footnotesize  24 }};
\draw  (5, -3)node (a34){{\footnotesize  34 }};
\draw (5, -4) node (a44){{\footnotesize  44. }};
\draw[mono] (a00)--(a01);
\draw[mono] (a11)--(a12);
\draw[mono] (a01)--(a02);
\draw[mono] (a02)--(a03);
\draw[mono] (a12)--(a13);
\draw[mono] (a22)--(a23);
\draw[mono] (a03)--(a04);
\draw[mono] (a13)--(a14);
\draw[mono] (a23)--(a24);
\draw[mono] (a33)--(a34);
\draw[epi] (a01)--(a11);
\draw[epi] (a12)--(a22);
\draw[epi] (a02)--(a12);
\draw[epi] (a03)--(a13);
\draw[epi] (a13)--(a23);
\draw[epi] (a23)--(a33);
\draw[epi] (a04)--(a14);
\draw[epi] (a14)--(a24);
\draw[epi] (a24)--(a34);
\draw[epi] (a34)--(a44);
\end{scope}

\begin{scope}[yshift=-0.3cm]
  \draw[twoarrowlonger] (2.2,0.1)--(2.8,-0.5);
  \draw[twoarrowlonger] (3.2,0.1)--(3.8,-0.5);
  \draw[twoarrowlonger] (3.2,-0.9)--(3.8,-1.5);
  \draw[twoarrowlonger] (4.2,0.1)--(4.8,-0.5);
  \draw[twoarrowlonger] (4.2,-1.9)--(4.8,-2.5);
  \draw[twoarrowlonger] (4.2,-0.9)--(4.8,-1.5);
\end{scope}
\end{tikzpicture}
\end{equation}
The connection can be made more precise, in particular taking the augmentation into account, via the augmented nerve functor. 
Considering $\wW{n}$ as a discrete augmented stable bisimplicial space, one can show that for a preaugmented bisimplicial space $Y$, we can identify
\[ \sdotn{n}(Y)\cong \Map_{\sasset}(\wW{n},Y). \]
\end{ex}

\section{The $\sdot$-construction for exact categories} \label{exactcategories}

In this section, we first define a functor which takes an exact category to an augmented stable double Segal groupoid, which we call the exact nerve.  Then we show that applying the generalized $\sdot$-construction to this output recovers more familiar $\sdot$-constructions for exact categories.

We use Keller's reformulation \cite[\S A.1]{KellerCh} 
of Quillen's original definition of an exact category \cite[\S 2]{QuillenK}; we refer the reader to both references for more details. 

\begin{defn}
An \emph{exact category} consists of an additive category $\exactcat$ together with a family $\cS$ of exact sequences 
\[ A \stackrel{i}{\mono} B \stackrel{p}{\epi} C. \]
We refer to the map $i$ as an \emph{admissible monomorphism} and the map $p$ as an \emph{admissible epimorphism}, subject to the following axioms.
\begin{enumerate}
\item The family $\cS$ is closed under isomorphisms.
    
\item The identity map of a zero object $\id_0$ is an admissible epimorphism. 
    
\item Admissible epimorphisms are closed under composition.
    
\item For any morphism $f\colon C\to C'$ in $\exactcat$ and any admissible epimorphism $p'\colon B' \to C'$, the pullback
\[ \tikzcdset{arrow style=tikz, diagrams={>=stealth'}}
    \begin{tikzcd}
     B \arrow[d, "p" swap] \arrow[r] & B'\arrow[d, "p'", epi]\\
     C \arrow[r, "f" swap] & C'
    \end{tikzcd}
    \]
    exists and the map $p$ is an admissible epimorphism.
    
    \item For any morphism $g\colon A\to A'$ in $\exactcat$ and any admissible monomorphism $i\colon A \to B$, the pushout
    \[ \tikzcdset{arrow style=tikz, diagrams={>=stealth'}}
    \begin{tikzcd}
     A \arrow[d, "g" swap] \arrow[r, mono, "i"] & B\arrow[d]\\
     A' \arrow[r, "i'" swap] & B'
    \end{tikzcd} \]
exists and the map $i'$ is an admissible monomorphism.
\end{enumerate}
We denote the subcategories of admissible monomorphisms and epimorphisms by $\cM$ and $\cE$, respectively.

An {\em exact functor} is an additive functor between exact categories which preserves exact sequences. We denote the category of exact categories and exact functors by $\excat$. 
\end{defn}

One can also formulate the axioms of an exact category in terms of the subcategories $\cM$ and $\cE$, which makes them easier to generalize to a homotopical version, as in \cite{barwickq}; we discuss related ideas further in \cref{exactinfinitycategories}. When no confusion arises, we denote an exact category $(\exactcat, \cS)$ simply by~$\exactcat$.

\subsection{The exact nerve}

Throughout this section, we assume that $\exactcat$ is an exact category and define a nerve-type construction which takes $\exactcat$ to an augmented stable double Segal groupoid. For any category $\exactcat$, we denote by $\exactcat_{\text{iso}}$ its maximal subgroupoid.

\begin{defn}
The \emph{exact nerve} of an exact category $\exactcat$ is the preaugmented bisimplicial groupoid $N^e\exactcat\colon \Sigma^{\op}\to \gpd$ defined as follows.
\begin{enumerate}
\item The augmentation $(N^e\exactcat)_{-1}$ is the full subgroupoid of $\exactcat_{\textrm{iso}}$ span\-ned by the zero objects.

\item \label{GridsBisimplices}
The groupoid in degree $(\inda,\indb)$ is the full subgroupoid $(N^e\exactcat)_{\inda, \indb}$ of $\Fun([\inda]\times[\indb],\exactcat)_{\textrm{iso}}$ consisting of $(\inda\times \indb)$-grids in $\exactcat$ in which all horizontal morphisms are in $\cM$, all vertical morphisms are in $\cE$, and all squares are bicartesian.
\end{enumerate}
The bisimplicial structure is induced by the bi-cosimplicial structure on the collection of categories $[\inda]\times[\indb]$ for $\inda, \indb \geq 0$.  The additional map $(N^e\exactcat)_{-1}\to (N^e\exactcat)_{0,0}$ is the canonical inclusion of the groupoid of zero objects into the maximal subgroupoid of $\exactcat$.

Observe that since an exact functor preserves the collections $\cM$ and $\cE$ and the bicartesian squares, the exact nerve can be defined on exact functors and defines a functor
$$N^e\colon \excat \longrightarrow \sagpd.$$
\end{defn}

Let us look at \eqref{GridsBisimplices} in more detail. Objects in $(N^e\exactcat)_{\inda, \indb}$ consist of diagrams $F\colon[\inda]\times[\indb]\to\exactcat$ such that:
\begin{itemize}
\item for any $0\leq i\leq \inda$ and $0\leq j\leq k \leq \indb$, the morphism $F(i,j)\to F(i,k)$
is in $\cM$; 

\item for any $0\leq i\leq k\leq  \inda$ and $0\leq j\leq \indb$, the morphism $F(i,j) \to F(k,j)$ is in $\cE$; and
\item for any $0\leq i\leq k\leq  \inda$ and $0\leq j\leq \ell\leq \indb$, the square
\begin{center}
 \begin{tikzpicture}
\def\l{1cm}
  \begin{scope}
  \draw (0,0) node (a00){$F(i,j)$};
\draw (2.5*\l,0) node (a01){$F(i,\ell)$};
\draw (0,-\l) node (a10){$F(k,j)$};
\draw (2.5*\l,-\l) node (a11){$F(k,\ell)$};
\draw[mono] (a00)--(a01);
\draw[epi] (a01)--(a11);
\draw[mono] (a10)--(a11);
\draw[epi] (a00)--(a10);
  \end{scope}
  \end{tikzpicture}
  \end{center}
  is bicartesian.
\end{itemize}

The objects of this groupoid can be pictured as diagrams in $\exactcat$ of the form
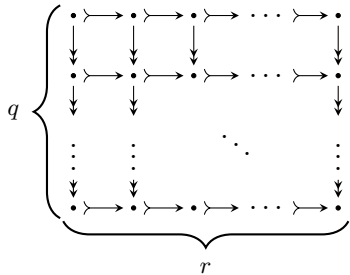
\begin{center}
    \begin{tikzpicture}[scale=0.8]
     \draw[fill] (0,0) circle (1pt) node (a00){};
     \draw[fill] (1,0) circle (1pt) node (a01){};
     \draw[fill] (2,0) circle (1pt) node (a02){};
     \draw (3.2,0) node (a03){$\ldots$}; 
     \draw[fill] (4.4,0) circle (1pt) node (a04){};
     
          \draw[fill] (0,-1) circle (1pt) node (a10){};
     \draw[fill] (1,-1) circle (1pt) node (a11){};
     \draw[fill] (2,-1) circle (1pt) node (a12){};
     \draw (3.2,-1) node (a13){$\ldots$}; 
     \draw[fill] (4.4,-1) circle (1pt) node (a14){};
     
     \draw (0,-2.2) node (a20){$\vdots$};
     \draw (1, -2.2) node (a21){$\vdots$};
     \draw (4.4,-2.2) node (a24){$\vdots$};
     
     \draw (2.7, -2) node (d){$\ddots$};
     
     \begin{scope}[yshift=-2.2cm]
     \draw[fill] (0,-1) circle (1pt) node (an0){};
     \draw[fill] (1,-1) circle (1pt) node (an1){};
     \draw[fill] (2,-1) circle (1pt) node (an2){};
     \draw (3.2,-1) node (an3){$\ldots$}; 
     \draw[fill] (4.4,-1) circle (1pt) node (an4){};
     \end{scope}
     
     \draw[mono] (a00)--(a01);
     \draw[mono] (a01)--(a02);
     \draw[mono] (a02)--(a03);
    \draw[mono] (a03)--(a04);
    
    \draw[mono] (a10)--(a11);
     \draw[mono] (a11)--(a12);
     \draw[mono] (a12)--(a13);
    \draw[mono] (a13)--(a14);
    
      \draw[mono] (an0)--(an1);
     \draw[mono] (an1)--(an2);
     \draw[mono] (an2)--(an3);
    \draw[mono] (an3)--(an4);
    
    \draw[epi] (a00)--(a10);
    \draw[epi] (a01)--(a11);
    \draw[epi] (a02)--(a12);
    \draw[epi] (a04)--(a14);
    
    \draw[epi] (a10)--(a20);
    \draw[epi] (a20)--(an0);
    
    \draw[epi] (a11)--(a21);
    \draw[epi] (a21)--(an1);
    
    \draw[epi] (a14)--(a24);
    \draw[epi] (a24)--(an4);

    \draw[decorate, decoration= {brace, raise=5pt, amplitude=10pt, mirror}, thick] (a00.north)--(an0.south)node [midway,xshift=-0.8cm] {\footnotesize $\inda$};
     \draw[decorate, decoration= {brace, raise=5pt, amplitude=10pt, mirror}, thick] (an0.west)--(an4.east)node [midway,yshift=-0.8cm] {\footnotesize $\indb$}; 

     \draw (an4) node[anchor=north west, yshift=0.1cm] {$.$};
    \end{tikzpicture}
\end{center}

The remainder of this subsection contains the proof of the following theorem, which in particular shows that the exact nerve functor factors though the full subcategory of augmented stable double Segal groupoids. 

\begin{thm} \label{thm nerve 1}
The exact nerve $N^e(\exactcat)$ is an augmented stable double Segal groupoid.
\end{thm}

Since the conditions of being stable, augmented, and double Segal are given by conditions described in terms of homotopy pullbacks, we first need a good model for homotopy pullbacks of groupoids.  To this end, we recall the model structure on groupoids, which can be found in \cite[Theorem 2.1]{Hollander}.

\begin{prop}\label{ModelStructureGpd}
The category $\gpd$ of groupoids admits a model structure, often called the \emph{canonical model structure}, in which the weak equivalences are equivalences of categories.  The fibrations are the \emph{isofibrations}, which are functors that have the right lifting property with respect to the inclusions $[0]\hookrightarrow I$ of a point into the free isomorphism category 
\[ I = \left\{
\begin{tikzpicture}[baseline=(base)]
\path (0,0) node[dot] (l) {} +(.5,0) node {\scriptsize $\cong$} +(1,0) node[dot](r) {} +(0,-3pt) coordinate (base); \path (l) edge[bend left, arrowinline] (r); \path (r) edge[bend left, arrowinline] (l);
\end{tikzpicture}
\right\}. \]
\end{prop}

\begin{rmk} \label{rem sset grpd} 
A useful fact, which can be observed from the definition of isofibration, is that the inclusion of some of the connected components of a groupoid is in fact an isofibration.
\end{rmk}

Homotopy pullbacks in this model structure are often referred to as \emph{2-pullbacks}. We denote by $\cD \twotimes{\cF}\cG$ the 2-pullback of a diagram of groupoids
\[ \cD\longrightarrow\cF\longleftarrow\cG. \]
In particular, the $2$-pullback construction is invariant under equivalences of groupoids. 

The usual nerve construction for groupoids defines a right Quillen functor between the canonical model structure on $\gpd$ and the model structure on $\sset$ due to Quillen. Indeed, even more is true: both weak equivalences and fibrations are created in $\sset$, as shown in \cite[Lemma 2.4]{Hollander}.   

Hence, given a preaugmented bisimplicial groupoid $N^e\exactcat$ we can compose with the nerve functor to obtain the preaugmented bisimplicial space
\[ \Sigma^{\op} \xrightarrow{N^e\exactcat} \gpd \xrightarrow{N} \sset. \] 
Observe that the conditions for being double Segal, augmented, and stable are described using homotopy pullbacks. Using the properties of the nerve above, we can see that a preaugmented bisimplicial groupoid satisfies those conditions precisely when its geometric realization 
satisfies the same conditions in the context of preaugmented bisimplicial spaces. In particular, the same fibrancy properties continue to hold in the model structure for groupoids. 

With this close relationship between groupoids and simplicial sets in mind, we consider the following special case of \cite[Proposition 1.3.8]{DK}.

\begin{prop} \label{2hpullback}
For any diagram of groupoids $\cD\rightarrow\cF\leftarrow\cG$,
the nerve of the 2-pullback $\cD \times^2_{\cF}\cG$
yields a model for the homotopy pullback of the corresponding nerves.  In other words, the canonical map
\[ N(\cD \twotimes{\cF}\cG)\xrightarrow{\simeq} N\cD \htimes{N\cF}N\cG \]
is a weak equivalence of simplicial sets.
\end{prop}

With the previous proposition in hand, we can show that the homotopy pullbacks that we need to establish the double Segal, stability, and augmentation conditions can be modeled by strict pullbacks.  We first look at the homotopy pullback used in the augmentation condition. 
A similar argument appears in the proof of \cite[Proposition 3.16]{PenneyHall}. 

\begin{lem} \label{lemmaisofibration0}
The canonical map 
\[ (N^e\exactcat)_{-1} \longrightarrow (N^e\exactcat)_{0,0} \]
is an isofibration of groupoids. Therefore a strict pullback along this map is a model for the $2$-pullback; in particular, there is an equivalence of categories
$$(N^e\exactcat)_{1,0} \twotimes{(N^e\exactcat)_{0,0}} (N^e\exactcat)_{-1}\simeq (N^e\exactcat)_{1,0} \ttimes{(N^e\exactcat)_{0,0}} (N^e\exactcat)_{-1}.$$ 
\end{lem}

\begin{proof} 
The map $(N^e\exactcat)_{-1} \to (N^e\exactcat)_{0,0}$ is the inclusion of a connected component of the groupoid $(N^e\exactcat)_{0,0}$. By \cref{rem sset grpd} it is an isofibration, so by \cref{2hpullback} the pullback in the statement of the lemma is in fact a homotopy pullback. 
\end{proof}

We now look at a similar result which can be applied to the homotopy pullbacks defining the double Segal and stability conditions.  The proof relies on a technical argument similar to the proof of \cref{jtprop4.10}, but in the context of functors $\Sigma^{\op} \rightarrow \sset$ rather than $\Delta^{\op} \rightarrow \sset$.  We give more details of a very similar proof in a more general context in \cref{lemmaKanfibration}. 

\begin{lem} \label{lemmaisofibration}
For any injective map $\theta\colon [\inda] \to [\inda']$ in $\Delta$ the induced map
\[ (N^e\exactcat)_{\inda', \indb} \longrightarrow (N^e\exactcat)_{\inda, \indb} \]
is an isofibration of groupoids. In particular, for any $0\leq i \leq \inda$ there is an equivalence of categories
\[ (N^e\exactcat)_{i,\indb} \twotimes{(N^e\exactcat)_{0,\indb}} (N^e\exactcat)_{\inda-i,\indb}\simeq (N^e\exactcat)_{i,\indb} \ttimes{(N^e\exactcat)_{0,\indb}} (N^e\exactcat)_{\inda-i,\indb}. \]
\end{lem}  

\begin{rmk}  \label{model2pullback}
These two lemmas facilitate the verification that the exact nerve of an exact category $\exactcat$ is an augmented stable double Segal groupoid.  For example, when $\inda=2$, $i=1$, and $\indb$ is arbitrary, via the previous lemma we can conclude that the $2$-pullback
\[ (N^e\exactcat)_{1,\indb} \twotimes{(N^e\exactcat)_{0,\indb}} (N^e\exactcat)_{1,\indb} \]
which appears in the double Segal condition can be modelled by an actual pullback. 
\end{rmk}

With these results, we now proceed to prove the main result of this section, namely, that $N^e\exactcat$ is an augmented stable double Segal groupoid.

\begin{prop} \label{exactnervedoubleSegal}
The preaugmented bisimplicial groupoid $N^e\exactcat$ is double Segal.
\end{prop}

\begin{proof}
We need to show that the maps 
\begin{align}
(N^e\exactcat)_{\inda,\indb} &\longrightarrow (N^e\exactcat)_{i,\indb} \twotimes{(N^e\exactcat)_{0,\indb}} (N^e\exactcat)_{\inda-i, \indb} \quad\text{and} \label{nerve Segality}\\
(N^e\exactcat)_{\inda, \indb} &\longrightarrow (N^e\exactcat)_{\inda, i} \twotimes{(N^e\exactcat)_{\inda, 0}} (N^e\exactcat)_{\inda, \indb-i} \label{nerve Segality2}
\end{align}
are equivalences. For simplicity, we focus on the latter case for $\inda =1, \indb=2$; the other cases can be proved similarly.

Observe that the map
\[ (N^e\exactcat)_{1,2} \longrightarrow (N^e\exactcat)_{1,1} \ttimes{(N^e\exactcat)_{1,0}} (N^e\exactcat)_{1,1} \] 
sends a $(1\times 2)$-grid of bicartesian squares to the pair of of squares given therein:
  \begin{center}
 \begin{tikzpicture}[inner sep=0.2pt]
\def\l{1cm}
  \begin{scope}
  \draw (0,0) node (a00){$d_{00}$};
\draw (\l,0) node (a01){$d_{01}$};
\draw (0,-\l) node (a10){$d_{10}$};
\draw (\l,-\l) node (a11){$d_{11}$};
\draw (2*\l,0) node (a02){$d_{02}$};
\draw (2*\l,-\l) node (a12){$d_{12}$};
\draw[mono] (a00)--(a01);
\draw[epi] (a01)--(a11);
\draw[mono] (a10)--(a11);
\draw[epi] (a00)--(a10);
\draw[mono] (a01)--(a02);
\draw[mono] (a11)--(a12);
\draw[epi] (a02)--(a12);
  \end{scope}
 \draw[mapstikz2] (3*\l, -0.5*\l)--(4*\l, -0.5*\l);
  \draw (4.6*\l, -0.5*\l) node(lbracket){
$\left(\vphantom{
\begin{minipage}[t][0.5cm][t]{3cm}
\end{minipage}
} \right.$ };

     \begin{scope}[xshift=5*\l]
   \draw (0,0) node (a00){$d_{00}$};
 \draw (\l,0) node (a01){$d_{01}$};
\draw (0,-\l) node (a10){$d_{10}$};
 \draw (\l,-\l) node (a11){$d_{11}$};
 \draw[mono] (a00)--(a01);
 \draw[epi] (a01)--(a11);
 \draw[mono] (a10)--(a11);
 \draw[epi] (a00)--(a10);
   \end{scope}
 \draw (6.5*\l, -0.5*\l) node (comma){$,$};
  
  \begin{scope}[xshift=6*\l]
\draw (\l,0) node (a01){$d_{01}$};
\draw (\l,-\l) node (a11){$d_{11}$};
\draw (2*\l,0) node (a02){$d_{02}$};
\draw (2*\l,-\l) node (a12){$d_{12}$};
\draw[epi] (a01)--(a11);
\draw[mono] (a01)--(a02);
\draw[mono] (a11)--(a12);
\draw[epi] (a02)--(a12);
  \end{scope}

  \draw (8.4*\l, -0.5*\l) node(rbracket){$\left.
  \vphantom{
\begin{minipage}[t][0.5cm][t]{3cm}
\end{minipage}
}
  \right)$};
  \draw (rbracket) node[xshift=0.2*\l] (comma2){$.$};
  \end{tikzpicture}
  \end{center}

We construct an inverse map; showing that it is indeed an inverse equivalence is straightforward.  Since the composite of bicartesian squares is still bicartesian, the assignment that glues together two bicartesian squares with a common edge, represented as:
   \begin{center}
 \begin{tikzpicture}[inner sep=0.2pt]
\def\l{1cm}
 \begin{scope}[xshift=-6cm]
  \draw (4.6*\l, -0.5*\l) node(lbracket){$\left(  \vphantom{
\begin{minipage}[t][0.5cm][t]{3cm}
\end{minipage}
}\right.$};
    \begin{scope}[xshift=5*\l]
  \draw (0,0) node (a00){$d_{00}$};
\draw (\l,0) node (a01){$d_{01}$};
\draw (0,-\l) node (a10){$d_{10}$};
\draw (\l,-\l) node (a11){$d_{11}$};
\draw[mono] (a00)--(a01);
\draw[epi] (a01)--(a11);
\draw[mono] (a10)--(a11);
\draw[epi] (a00)--(a10);
  \end{scope}

\draw (6.5*\l, -0.5*\l) node (comma){$,$};
  
  \begin{scope}[xshift=6*\l]
\draw (\l,0) node (a01){$d_{01}$};
\draw (\l,-\l) node (a11){$d_{11}$};
\draw (2*\l,0) node (a02){$d_{02}$};
\draw (2*\l,-\l) node (a12){$d_{12}$};
\draw[epi] (a01)--(a11);
\draw[mono] (a01)--(a02);
\draw[mono] (a11)--(a12);
\draw[epi] (a02)--(a12);
  \end{scope}

 \draw (8.4*\l, -0.5*\l) node(rbracket){$\left.  \vphantom{
\begin{minipage}[t][0.5cm][t]{3cm}
\end{minipage}
}\right)$};
  \end{scope}

    \draw[mapstikz2] (3*\l, -0.5*\l)--(4*\l, -0.5*\l);
  
  \begin{scope}[xshift=5*\l]
  \draw (0,0) node (a00){$d_{00}$};
\draw (\l,0) node (a01){$d_{01}$};
\draw (0,-\l) node (a10){$d_{10}$};
\draw (\l,-\l) node (a11){$d_{11}$};
\draw (2*\l,0) node (a02){$d_{02}$};
\draw (2*\l,-\l) node (a12){$d_{12}.$};
\draw[mono] (a00)--(a01);
\draw[epi] (a01)--(a11);
\draw[mono] (a10)--(a11);
\draw[epi] (a00)--(a10);
\draw[mono] (a01)--(a02);
\draw[mono] (a11)--(a12);
\draw[epi] (a02)--(a12);
  \end{scope}
   \end{tikzpicture}
 \end{center}
gives a well-defined map
\[ (N^e\exactcat)_{1,1} \ttimes{(N^e\exactcat)_{1,0}} (N^e\exactcat)_{1,1} \longrightarrow (N^e\exactcat)_{1,2}.\]

By \cref{lemmaisofibration}, we conclude that the map \eqref{nerve Segality2} for $\inda =1, \indb=2$ is also a weak equivalence.  Using a similar argument for the other double Segal maps \eqref{nerve Segality}, we obtain that $N^e\exactcat$ is double Segal.
 \end{proof}

\begin{prop} \label{exactnervestable}
The preaugmented bisimplicial groupoid $N^e\exactcat$ is stable. 
\end{prop}

\begin{proof}
First, observe that the map
\[ (N^e\exactcat)_{1,1} \longrightarrow (N^e\exactcat)_{1,0} \ttimes{(N^e\exactcat)_{0,0}} (N^e\exactcat)_{0,1} \] 
that sends a bicartesian square to the cospan contained in it, depicted as
\begin{center}
 \begin{tikzpicture}[inner sep=0.2pt]
  \def\l{1.0cm}
    \begin{scope}[xshift=-0.2*\l]
  \draw (0,0) node (c00){$c_{00}$};
\draw (\l,0) node (c01){$c_{01}$};
\draw (0,-\l) node (c10){$c_{10}$};
\draw (\l,-\l) node (c11){$c_{11}$};
\draw[mono] (c00)--(c01);
\draw[epi] (c01)--(c11);
\draw[mono] (c10)--(c11);
\draw[epi] (c00)--(c10);
  \end{scope}
 
\draw[mapstikz2] (1.25*\l, -0.5*\l)--(1.75*\l,-0.5*\l);
 
\begin{scope}[xshift=2.2*\l]
 
\draw (\l,0) node (c01){$c_{01}$};
\draw (0,-\l) node (c10){$c_{10}$};
\draw (\l,-\l) node (c11){$c_{11}$};
\draw[epi] (c01)--(c11);
\draw[mono] (c10)--(c11);
  \end{scope}  
 \end{tikzpicture}
\end{center}
is an equivalence of groupoids.

Indeed, since in $\exactcat$ every cartesian square is also bicartesian, an inverse equivalence
\[ (N^e\exactcat)_{1,0} \ttimes{(N^e\exactcat)_{0,0}} (N^e\exactcat)_{0,1} \longrightarrow (N^e\exactcat)_{1,1} \]
can be defined by functorially choosing pullbacks for each cospan
\begin{center}
 \begin{tikzpicture}[inner sep=0.2pt]
  \def\l{1.3cm}

  \begin{scope}[xshift=0.4*\l]
\draw (1.5*\l,0) node (c01){$c_{01}$};
\draw (0.5*\l,-\l) node (c10){$c_{10}$};
\draw (1.5*\l,-\l) node (c11){$c_{11}$};
\draw[epi] (c01)--(c11);
\draw[mono] (c10)--(c11);
  \end{scope}  
  
 \draw[mapstikz2] (2.25*\l, -0.5*\l)--(2.75*\l,-0.5*\l);
 
    \begin{scope}[xshift=3.5*\l, yshift=0]
  \draw (0,0) node (c00){$c_{01}\ttimes{c_{11}} c_{10}$};
\draw (1.5*\l,0) node (c01){$c_{01}$};
\draw (0,-\l) node (c10){$c_{10}$};
\draw (1.5*\l,-\l) node (c11){$c_{11}.$};
\draw[mono] (c00)--(c01);
\draw[epi] (c01)--(c11);
\draw[mono] (c10)--(c11);
\draw[epi] (c00)--(c10);
  \end{scope}
 \end{tikzpicture}
\end{center}
By \cref{lemmaisofibration}, the map
\[ (N^e\exactcat)_{1,1} \longrightarrow (N^e\exactcat)_{1,0} \twotimes{(N^e\exactcat)_{0,0}} (N^e\exactcat)_{0,1} \]
is thus also a weak equivalence.  It can be proven similarly that the analogous map
\[ (N^e\exactcat)_{1,1} \longrightarrow (N^e\exactcat)_{0,1} \twotimes{(N^e\exactcat)_{0,0}} (N^e\exactcat)_{1,0} \]
is also an equivalence of groupoids.  Applying \cref{babystability}, we conclude that $N^e\exactcat$ is stable.
\end{proof}

\begin{prop}
The double Segal groupoid $N^e\exactcat$ is aug\-men\-ted.
\end{prop}

\begin{proof}
Consider the source map
\[ (N^e\exactcat)_{1,0} \ttimes{(N^e\exactcat)_{0,0}} (N^e\exactcat)_{-1} \longrightarrow (N^e\exactcat)_{1,0} \to (N^e\exactcat)_{0,0}.\]
We claim that this map is an equivalence of groupoids; an inverse equivalence is given by sending an object to the unique vertical arrow to a fixed zero object~$0$. To clarify, the original source map can be depicted by the assignment on the left, and the inverse equivalence by the one on the right:
\begin{center}
\begin{tikzpicture}[baseline=(c00.base)]
 \def\l{1.0cm}
 \begin{scope}
  \draw (-0.5*\l, -0.5*\l) node(lbracket){$\left(  \vphantom{
\begin{minipage}[t][0.5cm][t]{3cm}
\end{minipage}
}\right.$};
  
\draw (0,0) node (a00){$c_{00}$};

\draw (0,-\l) node (a10){$a_{10}$};
\draw[epi] (a00)--(a10);
\draw (0.5*\l, -0.5*\l) node (comma){$,$};
  
\draw (1.5*\l, -0.4*\l) node (aug){$a_{10}$};

\draw (1.8*\l, -0.5*\l) node(rbracket){$\left.  \vphantom{
\begin{minipage}[t][0.5cm][t]{3cm}
\end{minipage}
}\right)$};
  \end{scope}
  
 \draw[mapstikz2] (2.3*\l, -0.4*\l)--(3.3*\l, -0.4*\l);
     \begin{scope}[xshift=4*\l]
  \draw (0,-0.4*\l) node (c00){$c_{00}$,};
  \end{scope}

  \end{tikzpicture}
\hspace{0.5cm}
 \begin{tikzpicture}[baseline=(c00.base)]
  \def\l{1.0cm}
  \draw (0,0) node (c00){$c_{00}$};
  \draw[mapstikz2] (0.5*\l, 0)--(1.5*\l, 0);
\begin{scope}[xshift=2.5*\l, yshift=0.5*\l]
 \draw (-0.4*\l, -0.5*\l) node(lbracket){$\left(
   \vphantom{
\begin{minipage}[t][0.5cm][t]{3cm}
\end{minipage}
}
\right.$};
  
\draw (0,0) node (a00){$c_{00}$};

\draw (0,-\l) node (a10){$0$};
\draw[epi] (a00)--(a10);
\draw (0.5*\l, -0.5*\l) node (comma){$,$};
  
\draw (1.5*\l, -0.4*\l) node (aug){$0$};

\draw (1.8*\l, -0.5*\l) node(rbracket){$\left.
\vphantom{
\begin{minipage}[t][0.5cm][t]{3cm}
\end{minipage}
}\right)$};
  \end{scope}
  \draw (rbracket) node[xshift=0.2*\l] (comma2){$.$};
 \end{tikzpicture}
\end{center}

Once again using \cref{lemmaisofibration0}, we have that the map
\[ (N^e\exactcat)_{1,0} \twotimes{(N^e\exactcat)_{0,0}} (N^e\exactcat)_{-1} \to (N^e\exactcat)_{1,0} \longrightarrow (N^e\exactcat)_{0,0} \]
is also a weak equivalence.   By a similar argument the map in the other simplicial direction
\[ (N^e\exactcat)_{0,1} \twotimes{(N^e\exactcat)_{0,0}} (N^e\exactcat)_{-1} \to (N^e\exactcat)_{0,1} \longrightarrow (N^e\exactcat)_{0,0},\]
is also an equivalence of groupoids.  Using \cref{babyaugmentation} and \cref{exactnervestable}, we conclude that $N^e\exactcat$ is augmented.
\end{proof}

\subsection{Comparing $\sdot$-constructions for exact categories}
\label{Comparing Sdot-constructions for exact categories}

Now that we have established that $N^e\exactcat$ is an augmented stable double Segal groupoid, we can apply our generalized $\sdot$-construction from \cref{sdot}.  To justify that this construction is indeed a generalization, we want to show that the result of applying it to $N^e\exactcat$ agrees with the output of the classical $\sdot$-construction applied to the exact category $\exactcat$.

There are two possible constructions to which we can compare our new one.  Both make use of the category of arrows of $[n]$, denoted by $\operatorname{Ar}[n]:=\Fun([1],[n])$.  Roughly speaking, $\Ar[n]$ is an ordinary category that has the same shape as $\wW{n}$ from \cref{nerveofWn}.  Based on that depiction, we refer to morphisms in the ``horizontal" and ``vertical" directions, as well as to the objects along the diagonal, namely those indexed by $ii$ for some $i$.  We discuss the difference between $\Ar[n]$ and $\wW{n}$ in more depth in \cref{arnvswn} below.

Given an (exact) category $\exactcat$, we can look at the category $\Fun(\Ar[n], \exactcat)$ of $\Ar[n]$-diagrams in $\exactcat$ as well as its maximal subgroupoid $\Fun(\Ar[n], \exactcat)_{\textrm{iso}}$.  The first definition is the original one due to Waldhausen \cite[\S 1.9]{waldhausen}. 

\begin{defn}
The \emph{Waldhausen construction} of an exact category $\exactcat$ is the simplicial groupoid $\sdotw(\exactcat)$ defined as follows.  For each $n \geq 0$, $\sdotwn{n}(\exactcat)$ is the the full subgroupoid of $\Fun(\operatorname{Ar}[n],\exactcat)_{\textrm{iso}}$ consisting of diagrams of shape $\operatorname{Ar}[n]$ in $\exactcat$ in which all horizontal morphisms are in $\cM$, all vertical morphisms are in $\cE$, all squares are bicartesian, and any element of the diagonal of $\operatorname{Ar}[n]$ is mapped to a specified zero object 0 of $\exactcat$.
\end{defn}

The second $\sdot$-construction, which we denote by $\sdote(\exactcat)$, has been described by a number of authors, including Dyckerhoff and Kapranov in \cite[\S 2.4]{DK}, G\'alvez-Carrillo, Kock, and Tonks in \cite[\S 10.7]{GalvezKockTonks}, and Lurie in \cite[Remark 1.2.2.5]{LurieHA}.
Essentially, the definition is the same as Waldhausen's but we allow for the elements along the diagonal to be any zero object of $\exactcat$, not a specified one.

While it is not surprising that one can allow for the use of multiple zero objects, we give a formal proof of the comparison between the two constructions.

\begin{prop} \label{comparisonsdotexactclassical}
Let $\exactcat$ be an exact category. There is a levelwise equivalence of groupoids
\[ \sdotw(\exactcat)\stackrel{\simeq}{\longrightarrow}\sdote(\exactcat). \]
\end{prop}

\begin{proof}
We build a canonical map $\sdotw\exactcat\to\sdote\exactcat$ that is a levelwise equivalence.  If $\aug(\exactcat)$ denotes the groupoid of zero objects of $\exactcat$ and $0$ is a fixed zero object, there is an isomorphism of groupoids
\begin{equation} \label{waldsn} \sdotwn{n}(\exactcat)\cong\{0\}^{n+1}\ttimes{\aug(\exactcat)^{n+1}}\sdoten{n}(\exactcat).
\end{equation}
For each $0\leq \inda \leq n$ there is a map $\sdoten{n}(\exactcat)\to\aug(\exactcat)$, given by evaluating a diagram of shape $\operatorname{Ar}[n]$ at the object $(\inda,\inda)$, which is an isofibration. It follows that the pullback \eqref{waldsn} is a 2-pullback, and the map
\[ \sdotwn{n}(\exactcat)\cong\{0\}^{n+1}\ttimes{\aug(\exactcat)^{n+1}}\sdoten{n}(\exactcat)\longrightarrow\aug(\exactcat)^{n+1}\ttimes{\aug(\exactcat)^{n+1}}\sdoten{n}(\exactcat)\cong \sdoten{n}(\exactcat) \]
induced by the categorical equivalence $\{0\}\hookrightarrow\aug(\exactcat)$ is an equivalence of groupoids, as desired.
\end{proof}

The following theorem shows that the generalized $\sdot$-construction recovers the previous constructions for exact categories.

\begin{thm} \label{comparisonsdotexact} 
Let $\exactcat$ be an exact category.  There is an isomorphism of groupoids
$$\sdote(\exactcat)\cong\sdot(N^e\exactcat).$$
\end{thm}

\begin{rmk} \label{arnvswn} 
The difference between $\Ar[n]$ and $\wW{n}$ (depicted in \eqref{PictureW4}), which have essentially the same shape, comes to the fore in the course of proving this theorem, so it is important to distinguish the different contexts in which they live. An important heuristic difference between $\operatorname{Ar}[n]$ and $\wW{n}$ is the following. In the former, we consider the diagram as a category, with all arrows being morphisms in a common category, and the depiction merely helps to organize the data. In the latter, there are two distinct kinds of morphisms, horizontal and vertical, for which the depiction as such is an essential feature; for example, a horizontal and a vertical morphism cannot be composed with one another. In particular, the nerve $N\operatorname{Ar}[n]$ has the structure of a (discrete) simplicial space, whereas $\wW{n}$, as given in \cref{nerveofWn}, is a (discrete) preaugmented bisimplicial space. The difference becomes key in our comparison here, since $\sdoten{n}(\exactcat)$ is given by diagrams indexed by the category $\operatorname{Ar}[n]$, whereas $\wW{n}$ is more closely related to the $\sdot$-construction of an augmented stable double Segal object.  The next lemma identifies the correct framework in which to compare the two indexing shapes $\operatorname{Ar}[n]$ and~$\wW{n}$.
\end{rmk}

The following technical lemma is the main tool in the argument of the proof of \cref{comparisonsdotexact}, and is formulated in the setting of $\sset$ endowed with the Joyal model structure, recalled in \cref{joyalmodel}.  We denote by $\tau_1\colon\sset\to\cat$ the left adjoint to the nerve functor, often called the \emph{fundamental category functor} \cite[\S 1]{joyalquasicategories}. We further explain the intuition behind the functor $T$ in \cref{TNC}.

\begin{lem} \label{midanodyne}
Consider the functor $T\colon\sset\to \sasset$ defined by
\[ (TX)_{\inda,\indb}:= \Map_{\sset}(\Delta[\inda]\times \Delta[\indb], X)\text{ and } (TX)_{-1}:= \Map_{\sset}(\Delta[0], X). \]
\begin{enumerate}
\item[(a)] 
The functor $T$ is part of a simplicial adjoint pair
\[ L\colon \sasset \leftrightarrows \sset\colon T. \]

\item[(b)] 
There is a natural acyclic cofibration
\[ L\wW{n}\longrightarrow N\operatorname{Ar}[n]. \]

\item[(c)] The acyclic cofibration from (b) induces an isomorphism of categories
\[ \tau_1(L\wW{n})\longrightarrow\tau_1(N\operatorname{Ar}[n])\cong \operatorname{Ar}[n]. \]

\item[(d)] 
If $f \colon X\longrightarrow Y$ is a map of simplicial sets inducing an isomorphism of categories $\tau_1(f)\colon \tau_1(X)\to\tau_1(Y)$, then for any category $\cD$ the map $f$ induces an isomorphism of mapping spaces
\[ \Map_{\sset}(Y,N\cD)\to \Map_{\sset}(X,N\cD). \] 
\end{enumerate} 
\end{lem}

We postpone the proof to the appendix, where we also recall more of the main ingredients involved, including background on quasi-categories.

\begin{rmk} \label{TNC}
Applying the functor $T$ to the ordinary nerve $N\cD$ of a category $\cD$, one can verify that $TN\cD$ is the preaugmented bisimplicial simplicial set described as follows.
\begin{enumerate}
\item The augmentation is simply the nerve of $\cD$:
\[ (TN\cD)_{-1}=N\cD. \]

\item \label{GridsBisimplicesTNC}
For each $(\inda,\indb)$, we recover the nerve of the category of $(\inda\times\indb)$-grids in $\cD$:
\[ (TN\cD)_{\inda, \indb}=N\Fun([\inda]\times[\indb], \cD). \]
\end{enumerate}
The bisimplicial structure is induced by the bi-cosimplicial structure on the collection of categories $[\inda]\times[\indb]$ for $\inda, \indb \geq 0$. The additional map $(TN\cD)_{-1}\to (TN\cD)_{0,0}$ is the identity map.
\end{rmk}

We now establish our desired comparison between $\sdote(\exactcat)$ and~$\sdot(N^e\exactcat)$.

\begin{proof}[Proof of \cref{comparisonsdotexact}]
The isomorphism of categories 
\[ \tau_1(L\wW{n})\xrightarrow{\cong}\tau_1(N\operatorname{Ar}[n])\cong \operatorname{Ar}[n] \]
from \cref{midanodyne}(c) induces, via \cref{midanodyne}(d) and \cref{midanodyne}(a), isomorphisms of simplicial sets
\begin{gather}\label{largeridentification}
\begin{aligned}
N\Fun(\operatorname{Ar}[n],\exactcat)&\cong \Map_{\sset}(N\operatorname{Ar}[n], N\exactcat)\\
&\cong \Map_{\sset}(L\wW{n}, N\exactcat)\\
&\cong \Map_{\sasset}(\wW{n}, TN\exactcat).
\end{aligned}
\end{gather}

In particular, because $N\Fun(\operatorname{Ar}[n],\exactcat)$ is a nerve of a category and hence a quasi-category, we can conclude that the simplicial set $\Map_{\sasset}(\wW{n}, TN\exactcat)$ is a quasi-category.

We observe that for every $n$ the nerve of the groupoid $\sdoten{n}(\exactcat)$ is the maximal Kan complex contained in $N\Fun(\operatorname{Ar}[n],\exactcat)$ spanned by the objects of $\sdoten{n}(\exactcat)$. On the other hand, taking advantage of the explicit description of  $TN\exactcat$ from \cref{TNC}, one can identify  
$\sdotn{n}(N^e(\exactcat))$ with its nerve, which is the maximal Kan complex contained in $\Map_{\sasset}(\wW{n}, TN\exactcat)$ spanned by the objects of $\sdotn{n}(N^e\exactcat)$.

By comparing the object sets of $\sdoten{n}(\exactcat)$ and $\sdotn{n}(N^e\exactcat)$, and implicitly applying the nerve functor to these groupoids, we see that the isomorphism from (\ref{largeridentification}), restricts along the inclusions
\[ \sdoten{n}({\exactcat})\subset N\Fun(\operatorname{Ar}[n],\exactcat)\text{ and }\sdotn n(N^e(\exactcat))\subset\Map_{\sasset}(\wW{n}, TN\exactcat) \]
to an isomorphism of simplicial sets
\[ \begin{tikzcd}
\sdoten{n}(\exactcat) \arrow[d, hook] \arrow[r, "\cong"] & \sdotn{n}(N^e\exactcat) \arrow[d, hook]\\
N\Fun(\operatorname{Ar}[n],\exactcat) \arrow[r, "\cong"]& \Map_{\sasset}(\wW{n}, TN\exactcat)
\end{tikzcd} \]
as desired. 
\end{proof}

\section{The $\sdot$-construction for stable $(\infty,1)$-categories} 
\label{stableinfinitycategories}

We now turn to variants of the $\sdot$-construction whose input is some kind of $(\infty,1)$-category; in this section we consider stable $(\infty,1)$-categories. 

We use the quasi-category model for $(\infty,1)$-categories, for which a complete account can be found in \cite{Joyal} or \cite{htt}. For any category $D$ and quasi-category $\cQ$ there exists a quasi-category $\cQ^{D}$ of $D$-shaped diagrams in $\cQ$ (see \cref{mappingqcat})
and a notion of \emph{limit} and \emph{colimit} for any of such diagrams \cite[1.2.13.4]{htt}. In particular, it makes sense to say when a square in $\cQ$ is \emph{cartesian} or \emph{cocartesian}, and when an object of $\cQ$ is \emph{initial} or \emph{terminal}.  In particular, as for ordinary categories, a \emph{zero object} is one which is both initial and terminal.

The theory of stable quasi-categories is developed by Lurie in \cite{LurieHA}, and we follow the definitions given there. We are using the reformulation of the definition for stability from \cite[Proposition 1.1.3.4]{LurieHA}. 

\begin{defn}
A quasi-category $\cQ$ is \emph{stable} if it has a zero object, all finite limits and colimits exist, and a square is cartesian if and only if it is cocartesian. 

There is a natural notion of functor between stable quasi-categories which preserves this structure, called an \emph{exact functor} \cite[\S 1.1.4]{LurieHA}; we thus have a category of stable quasi-categories which we denote by $\stqcat$.
\end{defn}

\begin{ex}
Let $\cM$ be a stable combinatorial simplicial model category, such as the category of spectra in the sense of stable homotopy theory or the category of chain complexes of modules over a ring \cite{ss}.  The underlying quasi-category associated to $\mathcal M$ as in \cite[Definition 1.3.4.15]{LurieHA} has the structure of a stable quasi-category, by means of \cite[Theorem 4.2.4.1 and Corollary 4.2.4.8]{htt}.  An alternate but equivalent approach is given by Lenz \cite{lenz}.  
\end{ex}

\subsection{The stable nerve}

As for exact categories, our first goal is to define a nerve functor whose input is a stable quasi-category and whose output is a preaugmented bisimplicial space; we want to show that this output is in fact an augmented stable double Segal space.

We begin with the definition of the stable nerve.  For any quasi-category $\cQ$, we denote by $J(\cQ)$ the maximal Kan complex spanned by its vertices, as in \cref{jq}. 

\begin{defn}
The \emph{stable nerve} $N^s\cQ$  of a stable quasi-category $\cQ$ is the preaugmented bisimplicial space $N^s\cQ\colon \Sigma^{\op}\to \sset$ defined as follows.
\begin{enumerate}
\item The space in degree $(\inda,\indb)$ is the maximal Kan complex
\[ (N^s\cQ)_{\inda,\indb}\subset J(\cQ^{[\inda]\times[\indb]}) \]
spanned by $(\inda\times \indb)$-grids in $\cQ$ with all squares bicartesian, i.e., diagrams
$F\colon[\inda]\times[\indb]\to \cQ$ such that
for all $0\leq i\leq k\leq  \inda$ and $0\leq j\leq l\leq \indb$, the square
\begin{center}
 \begin{tikzpicture}
\def\l{1cm}
  \begin{scope}
  \draw (0,0) node (a00){$F(i,j)$};
\draw (2.5*\l,0) node (a01){$F(i,\ell)$};
\draw (0,-\l) node (a10){$F(k,j)$};
\draw (2.5*\l,-\l) node (a11){$F(k,\ell)$};
\draw[->] (a00)--(a01);
\draw[->] (a01)--(a11);
\draw[->] (a10)--(a11);
\draw[->] (a00)--(a10);
  \end{scope}
  \end{tikzpicture}
  \end{center}
  is bicartesian.

\item The augmentation space is the maximal Kan complex
\[ (N^s\cQ)_{-1}\subset J(\cQ) \]
spanned by the zero objects of $\cQ$.
\end{enumerate}
The bisimplicial structure is induced by the bi-cosimplicial structure of the collection of categories of the form $[\inda]\times[\indb]$ for $\inda, \indb \geq 0$.  The additional map $N^s\cQ_{-1}\to N^s\cQ_{0,0}$ is the canonical inclusion of the Kan complex of zero objects into $J(\cQ)$. 

It is straightforward to check that the stable nerve can be defined on exact functors and hence defines a functor 
\[ N^s\colon\stqcat\longrightarrow \sasset. \]
\end{defn}

\begin{rmk} 
In light of the construction in the previous section, it is worth noting what is perhaps a conspicuous absence of the two different flavors of morphisms which are used for the horizontal and vertical directions.  The main idea is that in defining the stable nerve of a  stable quasi-category, we have the desired conditions on the squares, even with only one kind of morphism that fills both roles.  The distinction will be recovered in the still more general notion of (proto-)exact quasi-category in the next section.
\end{rmk}

For the rest of this section, assume that $\cQ$ be a stable quasi-category. We want to show that the stable nerve functor factors through the subcategory of stable augmented double Segal spaces, via the following theorem. 

\begin{thm}\label{thm nerve 2}
The stable nerve $N^s(\cQ)$ is an augmented stable double Segal space. 
\end{thm}

As for the exact nerve, before proving this theorem, we need some technical results that allow us to model all homotopy pullbacks in question by ordinary pullbacks.

\begin{lem} \label{lemmaKanfibration0}
The canonical map 
\[ (N^s\cQ)_{-1} \longrightarrow (N^s\cQ)_{0,0} \] 
is a Kan fibration. In particular, there is an equivalence of simplicial sets
\[ (N^s\cQ)_{1,0} \ttimes{(N^s\cQ)_{0,0}} (N^s\cQ)_{-1}\simeq (N^s\cQ)_{1,0} \htimes{(N^s\cQ)_{0,0}} (N^s\cQ)_{-1}. \]
\end{lem}

\begin{proof}
We observe that $(N^s\cQ)_{-1}$ is a connected component of $(N^s\cQ)_{0,0}$. Indeed, it is the full sub-Kan complex spanned by the vertices given by zero objects, which are all equivalent to one another by \cite[Proposition 1.2.12.9]{htt}.

Given that the map $(N^s\cQ)_{-1} \to (N^s\cQ)_{0,0}$ is the inclusion of a connected component, it is a Kan fibration and the pullback is in fact a homotopy pullback.
\end{proof}

\begin{lem} \label{lemmaKanfibration}
For any injective map $\theta\colon [\inda] \to [\inda']$ in $\Delta$ the induced map
\[ (N^s\cQ)_{\inda', \indb} \longrightarrow (N^s\cQ)_{\inda, \indb} \]
is a Kan fibration between Kan complexes.
In particular, for any $0\leq i \leq \inda$,
\[ (N^s\cQ)_{i,\indb} \ttimes{(N^s\cQ)_{0,\indb}} (N^s\cQ)_{\inda-i,\indb}\simeq (N^s\cQ)_{i,\indb} \htimes{(N^s\cQ)_{0,\indb}} (N^s\cQ)_{\inda-i,\indb}. \]
\end{lem}

\begin{proof}
A rather technical argument, similar to the proof of \cref{jtprop4.10}, shows that the bisimplicial space $N^s\cQ$ is Reedy fibrant.  Moreover, if we denote by $\Delta[\inda, \indb]$ the bisimplicial set represented by an object $(\inda, \indb)$ of $\Delta\times\Delta$, the map $\theta$ induces a cofibration of bisimplicial sets
\[ \Delta[\inda,\indb]\longrightarrow\Delta[\inda',\indb]. \]
It follows that the induced map
\[ (N^s\cQ)_{\inda', \indb}=\Map(\Delta[\inda',\indb],i^*N^s\cQ) \longrightarrow \Map(\Delta[\inda,\indb],i^*N^s\cQ)= (N^s\cQ)_{\inda, \indb} \]
is a Kan fibration, as desired.
\end{proof}

\label{modelhpullback}
Just as in the previous section, the previous two lemmas allow us to verify the conditions for being augmented, stable, and double Segal using ordinary pullbacks; see \cref{model2pullback}. With these results in hand, we are able to resume our proof that $N^s\cQ$ is an augmented stable double Segal space.

\begin{prop} \label{stablenervedoubleSegal}
The preaugmented bisimplicial space $N^s\cQ$ is double Segal.
\end{prop}

\begin{proof}
By \cref{lemmaKanfibration}, showing that the Segal map in the second simplicial direction is a weak equivalence is equivalent to showing that the map
\begin{equation}\label{stable Segal map}
(N^s \cQ)_{\inda, \indb} \longrightarrow (N^s \cQ)_{\inda, i} \ttimes{(N^s \cQ)_{\inda, 0}} (N^s \cQ)_{\inda, \indb-i}
\end{equation}
is a weak equivalence. We give the proof for $\indb=2$, $i=1$; the general case and the analogous maps in the other simplicial direction can be treated similarly.

The map \eqref{stable Segal map} a fits into a commutative square
\begin{equation}
\label{PullbackStableNerveSegal}
\begin{tikzcd}[baseline=(current  bounding  box.center)]
 (N^s \cQ)_{\inda,2} \arrow[r]\arrow[d, hook] & (N^s \cQ)_{\inda,1} \ttimes{(N^s \cQ)_{q,0}} (N^s \cQ)_{\inda,1} \arrow[d, hook]\\
 J(\cQ^{[\inda]\times[2]}) \arrow[r] & J(\cQ^{[\inda]\times[1]}) \ttimes{J(\cQ^{[\inda]\times[0]})} J(\cQ^{[\inda]\times[1]}).\end{tikzcd}
\end{equation}

This diagram can be shown to be a pullback square. For instance, when $\inda=1$, an element of $(N^s \cQ)_{1,2}$ essentially encodes two adjacent squares together with their composite such that all three squares are bicartesian.  On the other hand, an element of the pullback of the diagram
\[ J(\cQ^{[1]\times[2]}) \to J(\cQ^{[1]\times[1]}) \ttimes{J(\cQ^{[1]\times[0]})} J(\cQ^{[1]\times[1]}) \leftarrow (N^s \cQ)_{1,1} \ttimes{(N^s \cQ)_{1,0}} (N^s \cQ)_{1,1} \]
encodes precisely the data of two adjacent bicartesian squares and a choice of composite. By \cite[Lemma 4.4.2.1]{htt} the composite of bicartesian squares is automatically a bicartesian square. 

The bottom map of the square \eqref{PullbackStableNerveSegal} is a weak equivalence, as it can be identified with the Segal map 
\[ \Gamma(\cQ^{[\inda]})_2\xrightarrow{\simeq}\Gamma(\cQ^{[\inda]})_1\ttimes{\Gamma(\cQ^{[\inda]})_0}\Gamma(\cQ^{[\inda]})_1 \]
of the simplicial space $\Gamma(\cQ^{[\inda]})$ from \cref{gammadefn}, which is a Reedy fibrant Segal space by \cref{jtprop4.10}.  Moreover, for any $q$ the inclusion
\[ (N^s\cQ)_{q,\indb}\longrightarrow J(\cQ^{[q]\times[\indb]}) \]
is a Kan fibration, as it is the inclusion of a connected component. By iterating Reedy's Lemma (the dual to \cite[Lemma 7.2.15]{Hirschhorn}), 
we obtain that the right vertical map of \eqref{PullbackStableNerveSegal} is a Kan fibration.  Since Quillen's model structure on simplicial sets is right proper, it follows that the map \eqref{stable Segal map}
is a weak equivalence. We conclude that $N^s \cQ$ is double Segal.
\end{proof}

\begin{prop}
The preaugmented bisimplicial space $N^s\cQ$ is stable.
\end{prop}

\begin{proof}
By \cref{lemmaKanfibration}, showing that the stability map \eqref{eq:babystabilitycospan} is a weak equivalence is equivalent to showing that the map
\begin{equation}\label{eqn stability of Ns}
(N^s \cQ)_{1,1} \longrightarrow (N^s \cQ)_{1,0} \ttimes{(N^s \cQ)_{0,0}} (N^s \cQ)_{0,1}
\end{equation}
induced by $(\targetverd, \targethord)$ is a weak equivalence. We identify this map using the full sub-quasi-category of squares of $\cQ$ spanned by those that are bicartesian which we denote by 
\[ \widetilde{\cQ^{[1]\times[1]}}\subset \cQ^{[1]\times[1]}. \]
By \cite[Proposition 4.3.2.15]{htt} the restriction of the map
\begin{equation}
\label{mapstablebeforeJ}
\widetilde{\cQ^{[1]\times[1]}}\hookrightarrow \cQ^{[1]\times[1]}\longrightarrow \cQ^{[1]}\ttimes{\cQ^{[0]}}\cQ^{[1]},
\end{equation}
sending a square to the cospan that it contains, is a Joyal equivalence.
By \cref{jtprop1.16}, the functor $J$ commutes with pullbacks and sends Joyal equivalences of quasi-categories to weak equivalences of Kan complexes. Therefore the map
\[ J(\widetilde{\cQ^{[1]\times[1]}})\longrightarrow J(\cQ^{[1]}\ttimes{\cQ^{[0]}}\cQ^{[1]})\cong J(\cQ^{[1]})\ttimes{J(\cQ^{[0]})}J(\cQ^{[1]}), \]
obtained by applying $J$ to the composite (\ref{mapstablebeforeJ}) is a weak equivalence. Observe that this map is precisely the map \eqref{eqn stability of Ns}. The proof for the other stability map \eqref{eq:babystabilityspan} is similar.
\end{proof}

\begin{prop} \label{stablenerveaugmented}
The double Segal space $N^s\cQ$ is augmented.
\end{prop}

In the proof we make use of Joyal's fat slice construction from \cite[\S 10]{Joyal} or \cite[\S 2.4]{RiehlVerity2Cat}. Roughly speaking, given a quasi-category $\cQ$ and a $0$-simplex $x$, the fat slice $\cQ_{//x}$ over $x$, which is given by
\[ \cQ_{//x}:=\cQ^{[1]}\,{}^t\!\ttimes{\cQ}^x\Delta[0], \]
models the quasi-category of co-cones in $\cQ$ with vertex $x$, generalizing the classical notion of slice in an ordinary category.

\begin{proof} 
Fix a zero object $z$ of $\cQ$. Then the map $z\colon \Delta[0] \to (N^s \cQ)_{-1}$ is a weak equivalence.  This fact, together with \cref{lemmaKanfibration0}, implies that showing that the augmentation map \eqref{eqn: augmented 1A} is a weak equivalence is equivalent to showing that the map
\begin{equation}\label{eqn augmentation of Ns}
(N^s \cQ)_{1,0} \ttimes{(N^s \cQ)_{0,0}} \Delta[0] \longrightarrow (N^s \cQ)_{1,0} \longrightarrow (N^s \cQ)_{0,0}
\end{equation}
induced by the projection is a weak equivalence.  To establish this weak equivalence, consider the fat slice $\cQ_{//z}$. Note that there is a canonical identification of $J(\cQ_{//z})$ with the source of \eqref{eqn augmentation of Ns}, which can be seen by unpacking the construction of the fat slice $\cQ_{//z}$. By \cite[\S 10.3]{Joyal}, the source map
\begin{equation*}
s\colon\cQ_{//z}\longrightarrow \cQ^{[1]}\longrightarrow \cQ,
\end{equation*}
is a Joyal equivalence.  Hence, its image under $J$, which is precisely \eqref{eqn augmentation of Ns}, is a weak equivalence. The argument for the analogous map \eqref{eqn: augmented 2A} is similar.
\end{proof}

\subsection{Comparing $\sdot$-constructions for stable quasi-categories}

We now show that applying our generalized $\sdot$-construction to the stable nerve recovers the $\sdot$-construction for stable quasi-categories as considered by Lurie \cite[Remark 1.2.2.5]{LurieHA}, Blumberg, Gepner, and Tabuada \cite[Definition 7.1]{BGT}, and Barwick \cite{barwickq,barwickKtheory}. 

\begin{defn}
The \emph{Waldhausen construction} of a stable quasi-category $\cQ$ is the simplicial object $\sdots(\cQ)$ in simplicial sets defined as follows.  The $n$-th component $\sdotsn{n}(\cQ)$ is the sub-simplicial set of $J(\cQ^{\operatorname{Ar}[n]})$ spanned by diagrams of shape $\operatorname{Ar}[n]$ in $\cQ$ in which all squares are bicartesian, and any element of the diagonal of $\operatorname{Ar}[n]$ is mapped to a zero object of $\cQ$.  The simplicial structure is induced by the cosimplicial structure of the collection of categories $[n]$ for $n\geq 0$.
\end{defn}

We show compatibility of $\sdot$-constructions via the following result.

\begin{thm} \label{comparisonsdotstable}
Let $\cQ$ be a stable quasi-category. There is a levelwise acyclic Kan fibration
\[ \sdots(\cQ)\xrightarrow{\simeq}\sdot(N^s\cQ). \]
\end{thm}

The proof follows the same line of reasoning as \cref{comparisonsdotexact} but in a homotopical context.

\begin{proof}
The acyclic cofibration from \cref{midanodyne}(b) induces an acyclic fibration between quasi-categories
\[ \Map_{\sset}(N\operatorname{Ar}[n], \cQ)\xrightarrow{\simeq} \Map_{\sset}(L\wW{n}, \cQ)\cong\Map_{\sasset}(\wW{n}, T\cQ). \]

Given that the sets $\sdotsn{n}(\cQ)_0$ and $\sdotn{n}(N^s\cQ)_0$ are in canonical bijection, we see that this acyclic fibration restricts along the inclusions
\[ \sdotn n{\cQ}\subset \Map_{\sset}(\operatorname{Ar}[n],\cQ)\text{ and }\sdotn n(N^s(\cQ))\subset\Map_{\sasset}(\wW{n}, T\cQ). \]
The resulting commutative square
\[ \begin{tikzcd}
\sdotsn{n}(\cQ) \arrow[d, hook] \arrow[r] & \sdotn{n}(N^e\cQ) \arrow[d, hook]\\
\Map_{\sset}(N\operatorname{Ar}[n],\cQ) \arrow[r, "\simeq"]& \Map_{\sasset}(\wW{n}, T\cQ)
\end{tikzcd} \]
is a pullback, and in particular we obtain the desired acyclic fibration
\[ \sdotsn{n}(\cQ)\stackrel{\simeq}{\longrightarrow}\sdotn{n}(N^s\cQ). \qedhere \]
\end{proof}

\section{The $\sdot$-construction for proto-exact $(\infty,1)$-categories}
\label{exactinfinitycategories}

Barwick \cite{barwickq,barwickKtheory} and Dyckerhoff and Kapranov \cite{DK} introduce \emph{exact quasi-categories} and \emph{proto-exact quasi-categories}, notions which generalize both the definitions of exact categories and stable quasi-categories as considered in \cref{exactcategories,,stableinfinitycategories}. The notion of proto-exact quasi-category is general enough to include interesting examples that do not fit into the previous framework, such as sub-quasi-categories of a stable quasi-category which are closed under extensions, as considered in \cite{barwickKtheory}.

Here, we follow the definition of Dyckerhoff and Kapranov \cite[Definition 7.2.1]{DK} of an exact $\infty$-category, although we use the wording ``proto-exact'' rather than ``exact'' to distinguish from exact $\infty$-categories in the sense of Barwick \cite{barwickq}, which are in particular required to be additive.

Observe the similarity of terminology with exact categories, but also the use of the framework of quasi-categories.

\begin{defn} 
A \emph{proto-exact quasi-category} consists of a triple of quasi-categories $(\exactcat,\cM,\cE)$ such that:
\begin{itemize}
\item both $\cM$ and $\cE$ are sub-quasi-categories of $\exactcat$ containing all equivalences; 

\item the quasi-category $\exactcat$ has a zero object, $\cM$ contains all the morphisms whose source is a zero object, and $\cE$ contains all the morphisms whose target is a zero object; 

\item any pushout of a morphism in $\cM$ along a morphism of $\cE$ exists and belongs to $\cM$ and any pushout of a morphism in $\cE$ along a morphism of $\cM$ is in $\cE$; similarly, any pullback of a morphism in $\cM$ along a morphism of $\cE$ exists and belongs to $\cM$ and any pullback of a morphism in $\cE$ along a morphism of $\cM$ is in $\cE$; and

\item a square whose horizontal morphisms are in $\cM$ and whose vertical morphisms are in $\cE$ is cartesian if and only if it is cocartesian.
\end{itemize}
Using a suitable notion of exact functors which preserves this structure,
we obtain a category $\pexcat$ of proto-exact quasi-categories.
\end{defn}

In analogy with exact categories, we call the morphisms in $\cM$ and $\cE$ \emph{admissible monomorphisms} and \emph{admissible epimorphisms}, respectively. 

\begin{rmk}
It is likely that this definition could be extended to include the slightly more general {\em augmented} proto-exact $(\infty,1)$-categories of Penney \cite{PenneyHall}, which would produce an example of a stable augmented double Segal space with non-trivial augmentation.
\end{rmk}

As in the previous two sections, we begin by defining a suitable nerve functor for proto-exact quasi-categories.

\begin{defn} \label{nerveofexactinfinitycategories}
The \emph{proto-exact nerve} $N^{pe}\exactcat$ of a proto-exact quasi-category $\exactcat=(\exactcat,\cM,\cE)$, is the preaugmented bisimplicial space $N^{pe}\exactcat\colon \Sigma^{\op}\to \sset$ defined as follows.
\begin{enumerate}
\item The space in degree $(\inda,\indb)$ is the simplicial set
\[ (N^{pe}\exactcat)_{\inda,\indb}\subset J(\exactcat^{[\inda] \times [\indb]}). \]
spanned by $(\inda\times \indb)$-grids in $\exactcat$ with horizontal morphisms in $\cM$, the vertical morphisms in $\cE$, and all squares bicartesian. 

\item The augmentation space is the simplicial set
\[ (N^{pe}\exactcat)_{-1}\subset J(\exactcat) \]
spanned by all the zero objects of $\exactcat$.
\end{enumerate}
The bisimplicial structure is induced by the bi-cosimplicial structure on the collection of categories $[\inda]\times[\indb]$ for all $\inda, \indb \geq 0$. The additional map $(N^{pe}\exactcat)_{-1}\to (N^{pe} \exactcat)_{0,0}$ is the canonical inclusion of the space of zero objects into $\exactcat$. 

The proto-exact nerve defines a functor
\[ N^{pe}\colon\pexcat\longrightarrow \sasset. \]
\end{defn}

For the remainder of this section, assume that $\exactcat=(\exactcat,\cM,\cE)$ is a proto-exact quasi-category.

With a variant of the arguments from \cref{exactcategories,stableinfinitycategories}, one can prove that the proto-exact nerve factors through the subcategory of stable augmented double Segal spaces.

\begin{thm} \label{thm nerve 3}
The proto-exact nerve $N^{pe}(\exactcat)$ is an augmented stable double Segal space. 
\end{thm}

\begin{proof}[Outline of the proof]
The proof is a generalization of the one given for stable quasi-categories. Once again, the augmentation of $N^{pe}\exactcat$ is given by an inclusion of a connected component, and an argument similar to \cref{jtprop4.10} shows that $N^{pe}\exactcat$ is Reedy fibrant. We can thus conclude, as in \cref{lemmaKanfibration,lemmaKanfibration0}, that all homotopy pullbacks used to define the double Segal, augmentation, and stability conditions can be modeled by honest pullbacks. 

For an analogue of \cref{stablenervedoubleSegal}, in addition to the fact that bicartesian squares can be pasted in both directions \cite[Lemma 4.4.2.1]{htt}, we need the fact that both admissible monomorphisms and admissible epimorphisms are closed under composition by definition of proto-exact quasi-categories. This property shows that the Segal map 
\[ (N^{pe} \exactcat)_{\inda,\indb} \longrightarrow (N^{pe}\exactcat)_{\inda,\indb-i} \ttimes{(N^{pe} \exactcat)_{\inda,0}} (N^{pe} \exactcat)_{\inda,i},\] 
is the pullback of the Segal map
\[ \Gamma(\exactcat^{[\inda]})_\indb\xrightarrow{\simeq}\Gamma(\exactcat^{[\inda]})_{\indb-i}\ttimes{\Gamma(\exactcat^{[\inda]})_0}\Gamma(\exactcat^{[\inda]})_{i} \]
along a fibration, and therefore a weak equivalence.  

Similarly, both stability maps, for example
\[ (N^{pe} \exactcat)_{1,1} \longrightarrow (N^{pe}\exactcat)_{1,0} \ttimes{(N^{pe} \exactcat)_{0,0}} (N^{pe} \exactcat)_{0,1} \cong J(\cE^{[1]})\ttimes{J(\exactcat)}J(\cM^{[1]}), \]
are pullbacks of the map analogous to the one appearing in \eqref{mapstablebeforeJ}, restricting to diagrams with the necessary requirements on admissible monomorphisms and epimorphisms.  We use in particular the fact that morphisms in $\cE$ are closed under pullback along admissible monomorphisms and morphisms in $\cM$ are closed under pushout along admissible epimorphisms. 

Finally, the proof of the augmentation condition can be adapted from the proof of \cref{stablenerveaugmented}, using that every object of the augmentation is final in~$\cE$.
\end{proof}

The $\sdot$-construction for exact quasi-categories was considered by Barwick in \cite{barwickq,barwickKtheory}, and Dyckerhoff and Kapranov give a definition in the slightly more general context of proto-exact quasi-categories \cite{DK}.  We denote it by $\sdotpe(\exactcat)$ \cite[\S 2.4]{DK}. Once again, we have the following comparison result, whose proof is analogous to that of \cref{comparisonsdotstable}.

\begin{thm} \label{comparisonsdotexactinfinity}
If $\exactcat$ is a proto-exact quasi-category, there is a levelwise acyclic Kan fibration
\[ \sdotpe(\exactcat)\stackrel{\simeq}{\longrightarrow}\sdot(N^{pe}\exactcat). \]
\end{thm}

\begin{rmk} 
With a variant of \cref{nerveofexactinfinitycategories},  one could define the nerve of any Waldhausen category $\cW$, as introduced in \cite{waldhausen}, as a functor
\[ N^w\cW\colon\Sigma^{\op}\longrightarrow\gpd. \]
Informally, a Waldhausen category possesses the data of admissible monomorphisms which behave nicely under pushout, but does not have the dual notion of admissible epimorphism or the accompanying pullback data.  The preaugmented bisimplicial groupoid $N^w\cW$ can be shown to be double Segal by means of similar techniques to the ones employed in the proof of \cref{exactnervedoubleSegal}. However, it cannot be stable or augmented, because, a priori only half of the stablility and augmentation conditions hold.  For instance, a cospan given by a cofibration and an arbitrary morphism cannot be completed to a cartesian square in a canonical way, corresponding to the fact that the $\sdot$-construction of a Waldhausen category is not $2$-Segal, as discussed in \cite[Remark 7.3.7]{DK}.  As an explicit counterexample, we show in \cite[Example 4.3]{BOORS3} that the $\sdot$-construction of the Waldhausen category of retractive spaces over a point is not a 2-Segal space.
\end{rmk}

\section{The relative $\sdot$-construction}

In this section we show that the relative $\sdot$-construction, which is applied to a map of exact categories, fits into our framework.  Of particular interest is the fact that it provides examples with nontrivial augmentation, highlighting the full generality of that definition.  

Waldhausen's relative $\sdot$-construction, as defined in \cite{waldhausen}, is given in the more general context of Waldhausen categories.  However, we want to focus on contexts which correspond to $2$-Segal spaces. For simplicity we focus on maps between exact categories, but the results in this section apply to stable or proto-exact quasi-categories as well. 

To define the relative $\sdot$-construction let us first recall the initial path space construction of \cite{DK} and \cite{GalvezKockTonks}.   Let us denote by $-\star-\colon\Delta\times\Delta\to\Delta$ the join of the ordinals and $j \colon \Delta^{\op} \to \Delta^{\op}, [n]\mapsto [0] \star [n]$ the join with an initial object.

\begin{defn}
Let $\targetcat$ be a category.  The {\em initial path object} of a simplicial object $X \colon \Delta^{\op} \to \targetcat$ is the simplicial object $P^\triangleleft X$ given by the composite
\[ \Delta^{\op} \xrightarrow{j} \Delta^{\op} \xrightarrow{X} \targetcat. \]
\end{defn}

The initial path space comes with a canonical map induced by the extra face maps $d_0$ levelwise. By abuse of notation, we again denote this map by $d_0 \colon P^\triangleleft X \to X$. 

\begin{defn}
Let $f\colon\cA \to \cB$ be an exact functor between exact categories.
The {\em relative $\sdot$-construction of $f$}, denoted by $\sdotrel(f)$, is the pullback
of simplicial groupoids
\[  \begin{tikzcd}
\sdotrel(f) \arrow{r}\arrow{d} & \sdote (\cA) \arrow[d, "\sdote(f)"] \\
P^\triangleleft \sdote (\cB) \arrow[r, "d_0" swap] & \sdote (\cB).
\end{tikzcd} \]
\end{defn}

\begin{rmk} \label{pathfibration1}
For any exact category $\cB$, the map $d_0 \colon P^\triangleleft\sdote (\cB) \to \sdote (\cB)$ is a levelwise fibration of groupoids. In particular, the pullback defining the relative $\sdot$-construction is also a model for the $2$-pullback.
\end{rmk}

Let us look more closely at this definition. An object in $\sdotreln{n}(f)$
consists of an object $(a_{ij})$ in $\sdoten{n}(\cA)$ together with a sequence of admissible monomorphisms $b_{01}\mono \cdots \mono b_{0,n+1}$ in $\cB$ and admissible epimorphisms $b_{0j}\epi f(a_{1,j})$ which fit into a diagram  
\[ {\scriptsize \begin{tikzcd}[column sep=small]
0 \arrow[mono]{r} & b_{01} \arrow[epi]{d} \arrow[mono]{r} & b_{02} \arrow[mono]{r} \arrow[epi]{d} & b_{03} \arrow[mono]{r} \arrow[epi]{d} & \cdots \arrow[mono]{r} & b_{0n} \arrow[mono]{r} \arrow[epi]{d} & b_{0,n+1} \arrow[epi]{d}\\
& 0 \arrow[mono]{r} & f(a_{12}) \arrow[epi]{d} \arrow[mono]{r} & f(a_{13}) \arrow[mono]{r} \arrow[epi]{d} & \cdots \arrow[mono]{r} & f(a_{1n}) \arrow[mono]{r} \arrow[epi]{d} & f(a_{1,n+1}) \arrow[epi]{d}\\
& & 0 \arrow[mono]{r} & f(a_{23}) \arrow[mono]{r} & \cdots \arrow[mono]{r} & f(a_{2n}) \arrow[mono]{r} \arrow[epi]{d} & f(a_{2,n+1}) \arrow[epi]{d}\\
& & & & & \vdots \arrow[epi]{d} & \vdots \arrow[epi]{d} \\
 & & & & 0 \arrow[mono]{r} & f(a_{n-1,n}) \arrow[epi]{d} \arrow[mono]{r} & f(a_{n-1,n+1}) \arrow[epi]{d}\\
 & & & & & 0 \arrow[mono]{r} & f(a_{n,n+1}) \arrow[epi]{d}\\
 & & & & & & 0 \,,\\
\end{tikzcd} } \]
subject to the conditions that  all squares are bicartesian.

\begin{rmk}
Specializing the relative $\sdot$-construction to the constant functor $\cA\to\{0\}$ at the trivial exact category $\{0\}$, we obtain a levelwise acyclic fibration of groupoids
\[ \sdotrel\left(\cA\to\{0\}\right)\stackrel{\simeq}{\longrightarrow}\sdote(\cA). \]
\end{rmk}

We now show that the relative $\sdot$-construction arises directly from our generalized $\sdot$-construction. We first need a preliminary construction, which will be related to $P^\triangleleft \sdot (\cB)$. 

\begin{const} \label{const path object}
The {\em initial path construction} of a preaugmented bisimplicial object $Y$ in $\targetcat$ is the preaugmented bisimplicial object $P^\triangleleft Y \colon \Sigma^{\op} \to \targetcat$ defined as follows.
\begin{enumerate}
\item The underlying bisimplicial object of $P^\triangleleft Y$ is given by
\[ i^*(P^\triangleleft Y)\colon(\Delta\times\Delta)^{\op} \xrightarrow{j\times\id} (\Delta\times\Delta)^{\op} \xrightarrow{i^*Y} \targetcat; \]
in particular, for $\inda,\indb\geq0$, we have
\[ (P^\triangleleft Y)_{\inda,\indb} = (i^*Y)_{1+\inda, \indb}. \]

\item \label{Ptriangle condiction pullback} 
The augmentation of $P^\triangleleft Y$ is given by the pullback
\[ (P^\triangleleft Y)_{-1} = Y_{1,0} \ttimes{Y_{0,0}} Y_{-1}. \]
\end{enumerate}
The additional map $(P^\triangleleft Y)_{-1} \longrightarrow (P^\triangleleft Y)_{0,0}$ is given by the projection onto the first factor
\[ Y_{1,0} \ttimes{Y_{0,0}} Y_{-1} \to Y_{1,0}. \]
The initial path space defines a functor
\[ P^{\triangleleft}\colon\sasset\longrightarrow\sasset. \]
\end{const}

\begin{rmk}
If $\targetcat$ is a combinatorial model category and $Y$ is injectively fibrant, then the pullback in \eqref{Ptriangle condiction pullback} is also a homotopy pullback and hence homotopically well-defined. If in addition $Y$ is an augmented stable double Segal object, we have that the composite
\[ (P^\triangleleft Y)_{-1} \simeq Y_{1,0} \htimes{Y_{0,0}} Y_{-1} \longrightarrow Y_{0,0} \]
is a weak equivalence. For example, if we take $Y=N^e\cB$, then 
\[ (P^\triangleleft N^e\cB)_{-1} \simeq (N^e\cB)_{0,0} =  \cB_{\text{iso}} \]
as groupoids. In particular, even if the augmentation of $Y$ is trivial, that of $P^\triangleleft Y$ need not be. 
\end{rmk}

\begin{lem} \label{Ptrianglefibrant}
If $Y$ is an injectively fibrant augmented stable double Segal object in a combinatorial model category $\targetcat$, then $P^\triangleleft Y$ is an augmented stable double Segal object in $\targetcat$.
\end{lem}

\begin{proof}
To see that $P^\triangleleft Y$ is double Segal, we need to show that for fixed $\inda$ and $\indb$, the simplicial objects $Y_{1+\inda,\bullet}$ and $Y_{1+\bullet,\indb}$ are Segal objects. 

The fact that $Y$ is double Segal implies that $Y_{1+\inda,\bullet}$ is a Segal object.  For $Y_{1+\bullet,\indb}$, we can use the fact that 
\[ Y_{1+\bullet,\indb}=P^{\triangleleft}(Y_{\bullet,\indb}). \]
Since $Y_{\bullet,\indb}$ is a Segal object, it is also a $2$-Segal object. Hence, by the path criterion \cite[Theorem 6.3.2]{DK}, its path construction $P^{\triangleleft}(Y_{\bullet,\indb})$ is Segal.

To see that $P^\triangleleft Y$ is stable, consider the stability map from \eqref{eq:babystabilitycospan}, which by definition of $P^\triangleleft Y$ is the map
\begin{equation} \label{eqn stability for P}
Y_{2,1}\longrightarrow Y_{2,0}\htimes{Y_{1,0}}Y_{1,1}.
\end{equation}
This map takes an element in the left-hand side, which can be though of as two vertically composable squares, to the pair consisting of their horizontal targets in the first coordinate and the bottom square in the second.

In the following zigzag, the left-hand map can be seen to be an equivalence using a combination of stability and the Segal condition. Note that this map is exactly one of those appearing in the alternate definition of stability from \cref{adultstability}. The right-hand map is an equivalence by \eqref{eq:babystabilitycospan},
\[ Y_{2,1} \underset{\simeq}{\xrightarrow{(t_h, t_v)}} Y_{2,0}\htimes{Y_{0,0}}Y_{0,1} \simeq Y_{2,0}\htimes{Y_{1,0}}Y_{1,0} \htimes{Y_{0,0}} Y_{0,1} \underset{\simeq}{\xleftarrow{(t_h, t_v)}} Y_{2,0}\htimes{Y_{1,0}} Y_{1,1} \, .\]
By inspection, the map \eqref{eqn stability for P} fits into a commutative triangle with this zigzag of weak equivalences, hence is a weak equivalence itself.  A similar argument applies to the dual part of stability.

We now check that $P^\triangleleft Y$ is augmented. We first consider the augmentation map~\eqref{eqn: augmented 1A}, which, by definition of $P^\triangleleft Y$ is given by the composite
\[Y_{2,0} \prescript{d_1}{}{\htimes{Y_{1,0}}} (Y_{1,0} \htimes{Y_{0,0}} Y_{-1}) \xrightarrow{\pr_1} Y_{2,0} \xrightarrow{d_2} Y_{1,0}. \]
We can further factor it as the composite 
\[ Y_{2,0}\htimes{Y_{1,0}} Y_{1,0} \htimes{Y_{0,0}} Y_{-1} \simeq Y_{2,0} \htimes{Y_{0,0}} Y_{-1} \xrightarrow{\simeq} Y_{1,0}\htimes{Y_{0,0}} Y_{1,0} \htimes{Y_{0,0}} Y_{-1} \xrightarrow{\simeq} Y_{1,0}\htimes{Y_{0,0}} Y_{0,0} \xrightarrow{\simeq} Y_{1,0} \]
as follows.  The first equivalence is a general property of homotopy pullbacks, the second is given by the double Segal condition in the first variable, the third uses the associativity of homotopy pullbacks and augmentation, and the last is again a general property of homotopy pullbacks.

As for the augmentation map \eqref{eqn: augmented 2A}, by definition of $P^\triangleleft Y$ it is the map
\begin{equation} \label{eqn aug map Ptriangle}
Y_{1,1} \htimes{Y_{1,0}} (Y_{1,0} \htimes{Y_{0,0}} Y_{-1}) \xrightarrow{\pr_1} Y_{1,1} \longrightarrow Y_{1,0} \,.
\end{equation}
Consider the following diagram, of which both squares are homotopy pullbacks:
\[ \begin{tikzcd}
Y_{1,1} \htimes{Y_{1,0}} (Y_{1,0} \htimes{Y_{0,0}} Y_{-1}) \arrow[dashed]{r} \arrow[dashed]{d} & Y_{1,0}  \htimes{Y_{0,0}} Y_{-1} \arrow{d} \arrow{r} & Y_{-1} \arrow{d}\\
Y_{1,1} \arrow{r} & Y_{1,0} \arrow{r} & Y_{0,0}.
\end{tikzcd} \]
Hence, the outer square also is a homotopy pullback and the map \eqref{eqn aug map Ptriangle} can be identified with the map
\[ Y_{1,1} \htimes{Y_{1,0}} (Y_{1,0} \htimes{Y_{0,0}} Y_{-1}) \xrightarrow{\simeq} Y_{1,1} \htimes{Y_{0,0}} Y_{-1} \xrightarrow{\pr_1} Y_{1,1} \longrightarrow Y_{1,0} \,. \]
By inspection the composite of the two right-hand maps can be factored differently as follows. We can apply the stability map \eqref{eq:babystabilitycospan} in the first factor, then use associativity and the augmentation map \eqref{eqn: augmented 2A} in the second factor to obtain
\[ Y_{1,1} \htimes{Y_{0,0}} Y_{-1} \longrightarrow Y_{1,0} \htimes{Y_{0,0}} Y_{0,1} \htimes{Y_{0,0}} Y_{-1} \longrightarrow Y_{1,0} \htimes{Y_{0,0}} Y_{0,0} \simeq Y_{1,0}. \]
Since both are weak equivalences, the map \eqref{eqn aug map Ptriangle} is a weak equivalence as well.
\end{proof}

For any preaugmented bisimplicial object $Y$ the initial path construction comes with a canonical map induced by the extra face maps~$d_0$. By abuse of notation, we again denote this map by $d_0\colon P^\triangleleft Y \to Y$.

\begin{defn}
The {\em relative exact nerve} of an exact functor $f\colon \cA\to \cB$ is the preaugmented bisimplicial groupoid $\nrel f\colon \Sigma^{\op} \to \gpd$ given by the pullback
\[ \begin{tikzcd}
\nrel f \arrow{r} \arrow{d}& N^e\cA\arrow{d}{N^ef}\\
P^\triangleleft(N^e\cB) \arrow[r, "d_0" swap] & N^e\cB.
\end{tikzcd} \]
\end{defn}

\begin{rmk} \label{pathfibration2}
For an exact category $\cB$ and $Y=N^e\cB$, the map $d_0$ is a levelwise fibration of groupoids. In particular, the pullback defining the initial path object is also a model for the $2$-pullback.
\end{rmk}

\begin{rmk} \label{rmk preimage}
The relative exact nerve of an exact functor $f \colon \cA\to \cB$ is the preaugmented bisimplicial groupoid $\nrel f \colon \Sigma^{\op} \to \gpd$ defined as follows.
\begin{enumerate}
\item The augmentation $(\nrel f)_{-1}$ is the subgroupoid of $\Fun([1], \cB)_{iso}$ spanned by those admissible epimorphisms in $\cB$ whose target is a zero object.

\item An object in the groupoid in degree $(\inda, \indb)$ is an object $(a_{i,j})\in N^e(\cA)_{\inda, \indb}$ together with a sequence of admissible monomorphisms $b_{01}\mono \cdots \mono b_{0r}$ in $\cB$ and admissible epimorphisms $b_{0j}\epi f(a_{0j})$ which fit into a diagram 
\[ {\scriptsize \begin{tikzcd}
 b_{00} \arrow[epi]{d} \arrow[mono]{r} & b_{01} \arrow[mono]{r} \arrow[epi]{d} &  \cdots \arrow[mono]{r}   & b_{0 \indb} \arrow[epi]{d}\\
 f(a_{00}) \arrow[epi]{d} \arrow[mono]{r} & f(a_{01}) \arrow[epi]{d} \arrow[mono]{r} &  \cdots \arrow[mono]{r}   & f(a_{0 \indb}) \arrow[epi]{d}\\
  f(a_{10}) \arrow[epi]{d}\arrow[mono]{r} & f(a_{11}) \arrow[epi]{d} \arrow[mono]{r} &  \cdots \arrow[mono]{r}  & f(a_{1 \indb}) \arrow[epi]{d}\\
\vdots \arrow[epi]{d}& \vdots \arrow[epi]{d}  & \ddots & \vdots \arrow[epi]{d} \\
f(a_{\inda 0}) \arrow[mono]{r}  & f(a_{\inda 1}) \arrow[mono]{r}  & \cdots   \arrow[mono]{r} & f(a_{\inda, \indb})
\end{tikzcd} } \]
subject to the condition that all squares are bicartesian.
\end{enumerate}
\end{rmk}

\begin{rmk}
For any exact category $\cA$, the relative nerve of the functor $\cA\to\{0\}$, constant at the trivial exact category $\{0\}$, recovers the exact nerve of $\cA$. Indeed, there is a levelwise acyclic fibration of groupoids
\[ \nrel\left(\cA\to\{0\}\right)\stackrel{\simeq}{\longrightarrow} N^{e}(\cA). \]
\end{rmk}

\begin{thm} \label{prop relative asds}
The relative exact nerve of an exact functor $f\colon \cA\to \cB$ is an augmented stable double Segal groupoid.
\end{thm}

\begin{proof}
Given that $2$-pullbacks commute with each other, the augmentation, stability, and double Segal maps of $\nrel f$ can be expressed as the pullbacks of the corresponding maps for $N^e\cA$, $N^e\cB$ and $P^{\triangleleft}(N^e\cB)$. Given that these three preaugmented bisimplicial groupoids are stable augmented double Segal groupoids by \cref{thm nerve 1,Ptrianglefibrant}, it follows that $\nrel f$ is also.
\end{proof}

We can now apply our generalized $\sdot$-construction, and compare it to Waldhausen's original relative $\sdot$-construction.

\begin{thm} \label{thm comparison relative}
If $f\colon\cA\to\cB$ is an exact functor of exact categories, then there is a levelwise weak equivalence of simplicial groupoids
\[ \sdotrel f \longrightarrow \sdot(\nrel f). \]
\end{thm}

The proof of this theorem applies the following proposition to~$Y=N^{e}(\cB)$.

\begin{prop} \label{comparisoninitialpath}
If $Y$ is an augmented stable double Segal groupoid which is injectively fibrant as a preaugmented bisimplicial space, then there is levelwise acyclic fibration of simplicial groupoids
\[ P^\triangleleft(\sdot Y)\longrightarrow\sdot(P^\triangleleft Y). \]
\end{prop}

For formal reasons, the functor $P^{\triangleleft}\colon\sasset\to\sasset$ admits a left adjoint given by left Kan extension. This adjoint pair is simplicial and we denote the left adjoint by
\[ C^{\triangleleft}\colon\sasset\longrightarrow \sasset. \] 

\begin{proof}[Proof of \cref{comparisoninitialpath}]
We construct the desired map, first as a map of simplicial spaces
(rather than simplicial groupoids). A combinatorial argument
shows that $C^{\triangleleft}\wW n$ consists of a preaugmented double category obtained from $\wW n$ by adding a row of $n$ squares on top. For instance, $C^{\triangleleft}\wW 2$ could be depicted as
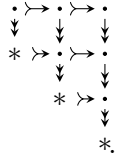
\begin{center}
\begin{tikzpicture}[scale=0.6]
\begin{scope}[xshift=12cm, yshift=-5cm]
\draw[fill] (2,0) circle (1pt) node(a01){};
\draw (2,-1) node(a11) {$*$};
\draw[fill] (3, -1) circle (1pt) node(a12){};
\draw[fill] (3, 0) circle (1pt) node(a02){};
\draw (3,-2) node (a22){$*$};
\draw[fill] (4, 0) circle (1pt) node (a03){};
\draw[fill] (4, -1) circle (1pt) node (a13){};
\draw[fill](4, -2) circle (1pt) node (a23){};
\draw (4, -3) node (a33){$*$};
\draw[mono] (a11)--(a12);
\draw[mono] (a01)--(a02);
\draw[mono] (a02)--(a03);
\draw[mono] (a12)--(a13);
\draw[mono] (a22)--(a23);

\draw[epi] (a01)--(a11);
\draw[epi] (a12)--(a22);
\draw[epi] (a02)--(a12);
\draw[epi] (a03)--(a13);
\draw[epi] (a13)--(a23);
\draw[epi] (a23)--(a33);
\draw (a33) node[xshift=0.1cm, yshift=-0.1cm]{$.$};
\end{scope}
\end{tikzpicture}
\end{center}
In particular, there is a canonical inclusion of discrete preaugmented bisimplicial spaces
\[ C^{\triangleleft}\wW n\hookrightarrow\wW{1+n}, \]
which includes $C^{\triangleleft}\wW n$ into $\wW{1+n}$ as the diagram without its $0$th column.  This inclusion can be checked to be an acyclic cofibration in the model structure for stable augmented double Segal spaces, being a pushout along an augmentation acyclic cofibration.

Note that $Y$, viewed as a preaugmented bisimplicial space, is fibrant in the model structure for stable augmented bisimplicial spaces.  Hence, we obtain an acyclic fibration of simplicial sets \cite[Proposition 9.3.1]{Hirschhorn}
\[ (P^\triangleleft(\sdot Y))_n = \Map(\wW{1+n},Y) \longrightarrow \Map(C^{\triangleleft}\wW n,Y)\cong\sdotn n P^\triangleleft Y. \]
The maps of simplicial sets that we constructed assemble to form a map of simplicial spaces
\[ P^\triangleleft(\sdot Y)\longrightarrow\sdot P^\triangleleft Y, \]
which is by construction a levelwise acyclic Kan fibration of simplicial sets. Using the fact that the nerve is fully faithful and creates weak equivalences and fibrations, we see that this map is in fact a levelwise acyclic fibration of groupoids, as desired.
\end{proof}

Given the emphasis on injective fibrancy of preaugmented bisimplicial spaces, we give the following useful criterion.

\begin{lem} \label{SigmaFibrancy}
Let $Y$ be a preaugmented bisimplicial space. Then $Y$ is injectively fibrant if and only if $i^*Y$ is a Reedy fibrant bisimplicial space and $Y_{-1}\to Y_{0,0}$ is a Kan fibration.
\end{lem}

\begin{proof}[Proof of \cref{thm comparison relative}]
We first observe that, by \cref{thm nerve 1}, $N^e(\cB)$ is a stable augmented double Segal groupoid, and one can use \cref{SigmaFibrancy} to show that it is injectively fibrant as a preaugmented bisimplicial space.
By \cref{comparisoninitialpath}, there is a levelwise equivalence of simplicial groupoids, and in particular a levelwise equivalence of simplicial spaces
\[ P^{\triangleleft}\sdote(\cB)=P^{\triangleleft}(N^e\cB)\longrightarrow\sdot P^{\triangleleft}(N^e\cB).\] 
This map fits into a commutative diagram together with two occurrences of isomorphisms from \cref{comparisonsdotexact}, as displayed:
\begin{equation}
\label{morphismofpullbacks}
\begin{tikzcd}
P^{\triangleleft}\sdote(\cB) \arrow{r} \arrow[two heads]{d}{\simeq} &\sdote(\cB) \arrow{d}{\cong} &\sdote(\cA) \arrow{l}\arrow{d}{\cong}\\
\sdot P^{\triangleleft}(N^e\cB) \arrow{r} &\sdot(N^e\cB)&\sdot (N^e\cA)\arrow{l}.
\end{tikzcd}
\end{equation}

Observe that the left-hand side horizontal maps are levelwise Kan fibrations. In \cref{pathfibration1}, the top one was observed to be a levelwise fibration of groupoids, and we now prove the claim for the bottom one. There is a canonical inclusion of discrete preaugmented bisimplicial spaces
\[ \wW n\hookrightarrow C^{\triangleleft}\wW n, \]
which includes $\wW n$ into $C^{\triangleleft}\wW{n}$ as the diagram without its $0$th row.

By mapping this inclusion into the injectively fibrant preaugmented bisimplicial space $N^e\cB$, we obtain a Kan fibration of simplicial sets
\[ \sdotn n P^{\triangleleft}(N^e\cB)\cong\Map(C^{\triangleleft}\wW n,N^e\cB)\longrightarrow\Map(\wW n,N^e\cB)\cong\sdotn n(N^e\cB). \]
The induced map between the pullbacks of the two rows is precisely a map of preaugmented bisimplicial spaces
\[ \sdotrel f\longrightarrow \sdot(P^{\triangleleft} N^e\cB) \ttimes{\sdot(N^e\cB)} \sdot (N^e\cA) \cong\sdot(\nrel f). \]
Observe that, given that the nerve is fully faithful, this map is in fact a map of preaugmented bisimplicial groupoids. We claim that each component
\[ \sdotreln{n} f = P^{\triangleleft}(\sdote\cB)_n\ttimes{\sdoten{n}{\cB}}\sdoten{n}{\cA}\longrightarrow \sdotn n(P^{\triangleleft}N^e\cB)\ttimes{\sdotn n {N^e\cB}}\sdotn n {N^e\cA} \cong \sdotn{n}(\nrel f), \]
is a weak equivalence. It is enough to recall that the vertical arrows in the diagram (\ref{morphismofpullbacks}) are levelwise acyclic Kan fibrations of simplicial sets, and the left-hand horizontal arrows are levelwise Kan fibrations. Given that the nerve functor creates weak equivalences, the map is an equivalence of groupoids, as desired.
\end{proof}

Using the fact that the generalized $\sdot$-construction of a stable augmented double Segal groupoid is a $2$-Segal groupoid, which was shown in \cite[Proposition 5.5]{BOORS2}, we immediately obtain the following corollary. Alternatively, it can be proven directly with a similar argument as for \cref{prop relative asds}.

\begin{cor}
The relative $\sdot$-construction of an exact functor $f \colon \cA \to \cB$ is a $2$-Segal groupoid.
\end{cor}

\appendix \label{appendix}

\stepcounter{section}
\section*{Appendix: Some results on quasi-categories} 
\renewcommand{\thethm}{\Alph{section}.\arabic{thm}}
\renewcommand{\thedefn}{\Alph{section}.\arabic{thm}}
\renewcommand{\theprop}{\Alph{section}.\arabic{thm}}
\renewcommand{\thermk}{\Alph{section}.\arabic{thm}}

Here we compile some results about quasi-categories and their model structure.

\begin{defn}
A \emph{quasi-category} is a simplicial set $\cQ$ such that a lift exists in any diagram of the form
\[ \xymatrix{\Lambda^k[n] \ar[r] \ar[d] & \cQ \\
\Delta[n] \ar@{-->}[ur]} \]
for any $n \geq 2$ and $0<k<n$.  A \emph{Kan complex} is a simplicial set $\cQ$ such that a lift exists in any such diagram for any $n \geq 1$ and any $0 \leq k \leq n$.
\end{defn}

\begin{defn} \label{mappingqcat} 
Let $\cQ$ be a quasi-category and $K$ any simplicial set.  Then we define $\cQ^K$ to be the internal hom
$$\cQ^K:=\Map_{\sset}(K,\cQ).$$
\end{defn}

\begin{prop}[{\cite[Proposition 1.2.7.3]{htt}}]
If $\cQ$ is a quasi-category and $K$ is any simplicial set, then $\cQ^K$ is a quasi-category.
\end{prop}

\begin{rmk}
If $\cD$ is a category, then we can make sense of $\cQ^{\cD}$ as $\cQ^{N\cD}$ via the above definition.  However, for simplicity we often omit the nerve notation and simply write $\cQ^\cD$.  In particular, we sometimes write $\cQ^{[n]}$ rather than $\cQ^{\Delta[n]}$.
\end{rmk}

There are a number of ways to formulate the following notion of weak equivalence between simplicial sets; the main idea is that it is meant to be a homotopical analogue of equivalence of ordinary categories.  Here, we use the definition of Joyal \cite[\S 2.5]{Joyal} 
and use the name \emph{Joyal equivalence} for what he calls \emph{weak categorical equivalence}; see \cite[Definition 1.1.5.14]{htt} for a different but equivalent approach, under the name of \emph{categorical equivalence}.

Recall the fundamental category functor $\tau_1 \colon \sset \rightarrow \cat$ which is left adjoint to the nerve.  Given a simplicial set $A$, we denote by $\tau_0(A)$ the set of isomorphism classes of objects of $\tau_1(A)$.

\begin{defn}  
A map $f \colon A \rightarrow B$ of simplicial sets is a \emph{Joyal equivalence} if, for any quasi-category $\cQ$, the induced map $\tau_0(\cQ^B) \rightarrow \tau_0(\cQ^A)$
is a bijection.
\end{defn}

The following theorem has been proved by Joyal \cite[Theorem 6.12]{JoyalVolumeII}, Lurie \cite[Theorem 2.2.5.1]{htt}, and Dugger and Spivak \cite[Theorem 2.13]{DuggerSpivakMapping}.

\begin{thm} \label{joyalmodel}
The category $\sset$ of simplicial sets admits a cartesian model structure called the \emph{Joyal model structure}, in which the weak equivalences are the Joyal equivalences, the cofibrations are the monomorphisms, and the fibrant objects are precisely the quasi-categories.
\end{thm}

The fibrations in this model structure are determined as the maps which have the right lifting property with respect to the monomorphisms which are Joyal equivalences, and we refer to them as \emph{quasi-fibrations}.

We also sometimes refer to the model structure on simplicial sets, originally due to Quillen.

\begin{thm}[\cite{QuillenHA}] \label{Quillenmodel} 
The category $\sset$ of simplicial sets admits a model structure in which the weak equivalences geometrically realize to weak homotopy equivalences of spaces, the cofibrations are the monomorphisms, and the fibrant objects are precisely the Kan complexes.
\end{thm}

We refer to the fibrations of this model structure as \emph{Kan fibrations}. 
In the context of quasi-categories, it is also useful to consider the following maps.

\begin{defn}
A map $f \colon K \rightarrow L$ of simplicial sets is:
\begin{enumerate}
    \item an \emph{inner fibration} if it has the right lifting property with respect to the maps $\Lambda^k[n] \rightarrow \Delta[n]$ for any $n\geq 2$ and $0<k<n$; and
    
    \item an \emph{inner anodyne map} if it has the left lifting property with respect to the inner fibrations.
\end{enumerate}
\end{defn}

While Kan fibrations are defined by an analogous lifting property but with respect to all horns, inner fibrations do not coincide with quasi-fibrations, in that the latter must also have the right lifting property with respect to the map $\Delta[0] \rightarrow NI$ where $I$ is the groupoid with two isomorphic objects.

The following can be seen as an application of Quillen's small object argument.

\begin{prop} \label{satclass}
The class of inner anodyne maps is precisely the smallest weakly saturated class containing the inner horns. In particular, any inner anodyne map can be built as a retract of transfinite compositions of pushouts of inner horn inclusions. 
\end{prop}

\begin{rmk} \label{inneranodynecofibration}
Given that all inner horns inclusions are cofibrations, and that cofibrations are closed under pushouts and transfinite compositions and retracts, all inner anodyne extensions are in particular cofibrations.
\end{rmk}

\begin{prop}[{\cite[Proposition 1.11]{JoyalTierney}, \cite[Lemma 2.2.5.2]{htt}}] \label{jtprop1.11} 
Any inner anodyne map is a Joyal equivalence which is bijective on vertices. Furthermore, the functor $\tau_1$ takes inner anodyne maps to isomorphisms of categories.
\end{prop}

\begin{defn}[{\cite[\S 1.10]{Joyal}}] \label{jq}
Let $\cQ$ be a quasi-category.  Let $J(\cQ)$ be the maximal Kan complex on the vertices of $\cQ$. 
\end{defn}

In fact, $J$ defines a functor from the category of quasi-categories to the category of Kan complexes.

\begin{prop}[{\cite[\textsection\textsection 4.2-4.3]{JoyalVolumeII}}]
\label{Jpreservesfibs} \label{jtprop1.16} 
The functor $J$ is right adjoint to the inclusion functor from the category of Kan complexes to the category of quasi-categories.  Furthermore, $J$ takes Joyal equivalences to weak equivalences and quasi-fibrations to Kan fibrations.
\end{prop}

\begin{defn}[{\cite[\S 4]{JoyalTierney}}] \label{gammadefn}
If $K$ is a simplicial set, then $\Gamma(K)$ is the simplicial space given by $\Gamma(K)_n = J(K^{[n]})$.
\end{defn}

The construction $\Gamma$ defines a functor $\Gamma \colon \sset \rightarrow \Fun(\Delta^{\op}, \sset)$.

We take the following definitions from Joyal and Tierney \cite[\S 2]{JoyalTierney}.  Consider the box product map
\[ - \square - \colon \sset \times \sset \rightarrow \Fun(\Delta^{\op}, \sset)\]
given by $(A, B) \mapsto \left( ([m],[n]) \mapsto A_m \times B_n \right)$.  Fixing the simplicial set $A$, we obtain a functor 
\[ A \square - \colon \sset \rightarrow \Fun(\Delta^{\op}, \sset) \]
which admits a right adjoint
\[ A \backslash - \colon \Fun(\Delta^{\op},\sset) \rightarrow \sset. \]
Observe that, for any simplicial space $X$, the simplicial set $\Delta[k] \backslash X$ is given by $X_{k,*}$.  Furthermore, $\partial \Delta[k] \backslash X$ agrees with the $k$-\emph{matching object} used to define the Reedy model structure on simplicial spaces \cite[Definition 15.2.6]{Hirschhorn}.

\begin{prop}[{\cite[Proposition 4.9]{JoyalTierney}}] \label{jtprop4.9}
If $X$ is a quasi-category and $A$ is any simplicial set, then 
\[ A \backslash \Gamma(X) = J(X^A). \]
In particular, $\Delta[k] \backslash \Gamma(X) = J(X^{[k]})$.
\end{prop}

We recall a sketch of the proof of the following proposition since we use the same technique several times in this paper.

\begin{prop}[{\cite[Proposition 4.10]{JoyalTierney}}] \label{jtprop4.10}
If $X$ is a quasi-category, then $\Gamma(X)$ is a Reedy fibrant Segal space.
\end{prop}

\begin{proof}
Let $X$ be a quasi-category, and consider the inclusion $\delta_n \colon \partial \Delta[n] \rightarrow \Delta[n]$.  The induced map \[ X^{\delta_n} \colon X^{\Delta[n]} \rightarrow X^{\partial \Delta[n]} \] 
is a quasi-fibration since the Joyal model structure is cartesian.  It follows that the map
\[ J(X^{\delta_n}) \colon J(X^{\Delta[n]}) \rightarrow J(X^{\partial \Delta[n]}) \]
is a Kan fibration by \cref{Jpreservesfibs}.  By \cref{jtprop4.9}, this map can be identified with
\[ \delta_n \backslash \Gamma(X) \colon \Delta[n] \backslash \Gamma(X) \rightarrow \partial \Delta[n] \backslash \Gamma(X). \]
This latter map is hence a Kan fibration, from which it follows that $\Gamma(X)$ is Reedy fibrant.

A similar argument using the inclusion $\mathcal{I}(n) \hookrightarrow \Delta[n]$ of the spine into the simplex which induces the Segal maps shows that $\Gamma(X)$ is a Segal space.
\end{proof}

Now we give a very brief treatment of ordinal subdivisions of simplicial sets and of categories, following work of Ehlers and Porter \cite{EhlersPorter}.

\begin{defn}
The \emph{ordinal subdivision} of the $n$-simplex $\Delta[n]$ is defined as the coend 
\[ Sd(\Delta[n]) := \int\limits^{p,q} \coprod_{[p+1+q] \to [n]} \Delta[p] \times \Delta[q]. \]
For a general simplicial set $K$, we then define the \emph{ordinal subdivision}
\[ Sd(K) := \int\limits^n \coprod_{\Delta[n] \to K} Sd(\Delta[n]). \]
\end{defn}

There is a similar definition for categories.

\begin{defn}
The \emph{ordinal subdivision} of a category $\cD$ is defined to be
\[Sd(\cD):= \int\limits^{p,q} \coprod_{[p+1+q] \to \cD} [p] \times [q]. \]
\end{defn}

We recall two useful results concerning these subdivisions. The first establishes the relation between these two definitions. Note that one could have hoped for the nerve of $Sd(\cD)$ to be precisely $Sd(N\cD)$, but this is far from being true in general. However, the following weaker, and yet very useful, relation holds. 

\begin{prop}[{\cite[Proposition 3.6]{EhlersPorter}}] \label{epprop3.6}
There is an isomorphism of categories
\[ \tau_1 Sd(\Delta[n]) \cong Sd [n]. \]
\end{prop}

For the second result, we need to recall the notion of ``weak anodyne extension", due to Ehlers and Porter \cite[\S 5.2]{EhlersPorter}; compare with \cref{satclass}.

\begin{defn} 
A map of simplicial sets is called a \emph{weak anodyne extension} if it can be expressed as a composite of finitely many maps which are pushouts of inner horn inclusions.
\end{defn}

\begin{thm}[{\cite[Theorem 6.1]{EhlersPorter}}] \label{epthm6.1}
For every $n \geq 0$, the unit map
\[ \eta \colon Sd \Delta[n] \rightarrow N Sd [n] \]
is a weak anodyne extension.
\end{thm}

We conclude with the deferred proof of a key lemma.

\begin{proof}[Proof of \cref{midanodyne}]
The simplicial adjunction in (a) is given by the theory of enriched Kan extensions as in \cite[Chapter 4]{Kelly}. 

We outline the key steps to construct the acyclic cofibration in (b), using the ordinal subdivision functor $Sd$ described above.
\begin{enumerate}
\item A careful analysis which makes use of \cref{epprop3.6} yields an isomorphism of categories
$$\operatorname{Ar}[n]=\Fun([1],[n])\cong Sd[n].$$

\item If $\delta\colon\Delta\to\Delta\times\Delta$ denotes the diagonal functor and $\sigma=p\circ i \colon \Delta \times \Delta \to \Delta$ denotes the ordinal sum on $\Delta$, the map of simplicial sets
$$\delta^*\sigma^*\Delta[n]\to NSd[n]$$
is the map from \cref{epthm6.1}, and therefore a Joyal equivalence by \cref{jtprop1.11}. By \cref{inneranodynecofibration}, the same map is in fact an acyclic cofibration in the Joyal model structure. 

\item Using \cref{rmkpstar(Delta)=W} and the formal properties of left adjoints, we obtain an isomorphism of simplicial sets
$$L\wW{n}\cong Lp^*\Delta[n]\cong\delta^*i^*p^*\Delta[n]= \delta^*\sigma^*\Delta[n].$$

\end{enumerate}
Combining (1), (2) and (3), we obtain the desired acyclic cofibration. The isomorphism in (c) follows from (b) and \cref{epprop3.6}. 

Part (d) is an immediate consequence of the fact that $(\tau_1, N)$ form a simplicial adjunction.
\end{proof}

\bibliographystyle{alpha}

\bibliography{bibliography}

\newcommand{\etalchar}[1]{$^{#1}$}
\begin{thebibliography}{BOO{\etalchar{+}}18b}

\bibitem[Bar13]{barwickq}
Clark Barwick.
\newblock On the \textnormal{Q}-construction for exact $\infty$-categories.
\newblock {\em arXiv:1301.4725}, 2013.

\bibitem[Bar16]{barwickKtheory}
Clark Barwick.
\newblock On the algebraic {$K$}-theory of higher categories.
\newblock {\em J.\ Topol.}, 9(1):245--347, 2016.

\bibitem[BGT13]{BGT}
Andrew~J. Blumberg, David Gepner, and Gon\c{c}alo Tabuada.
\newblock A universal characterization of higher algebraic {$K$}-theory.
\newblock {\em Geom. Topol.}, 17(2):733--838, 2013.

\bibitem[BOO{\etalchar{+}}18a]{BOORS2}
Julia~E. Bergner, Ang\'elica~M. Osorno, Viktoriya Ozornova, Martina Rovelli,
  and Claudia~I. Scheimbauer.
\newblock 2-{S}egal objects and the {W}aldhausen construction.
\newblock {\em arXiv:1809.10924}, 2018.

\bibitem[BOO{\etalchar{+}}18b]{BOORS}
Julia~E. Bergner, Ang\'{e}lica~M. Osorno, Viktoriya Ozornova, Martina Rovelli,
  and Claudia~I. Scheimbauer.
\newblock 2-{S}egal sets and the {W}aldhausen construction.
\newblock {\em Topology Appl.}, 235:445 -- 484, 2018.

\bibitem[BOO{\etalchar{+}}20]{BOORS3}
Julia~E. Bergner, Ang\'{e}lica~M. Osorno, Viktoriya Ozornova, Martina Rovelli,
  and Claudia~I. Scheimbauer.
\newblock The edgewise subdivision criterion for {$2$}-{S}egal objects.
\newblock {\em Proc. Amer. Math. Soc.}, 148(1):71--82, 2020.

\bibitem[{Car}18]{Carlier}
Louis {Carlier}.
\newblock {Incidence bicomodules, M\"obius inversion, and a Rota formula for
  infinity adjunctions}.
\newblock {\em arXiv:1801.07504}, 2018.

\bibitem[DK19]{DK}
Tobias {Dyckerhoff} and Mikhail {Kapranov}.
\newblock {\em {Higher Segal spaces.}}, volume 2244.
\newblock Cham: Springer, 2019.

\bibitem[DS11]{DuggerSpivakMapping}
Daniel Dugger and David~I. Spivak.
\newblock Mapping spaces in quasi-categories.
\newblock {\em Algebr. Geom. Topol.}, 11(1):263--325, 2011.

\bibitem[Ehr63]{Ehresmann}
Charles Ehresmann.
\newblock Cat\'egories structur\'ees.
\newblock {\em Ann. Sci. \'Ecole Norm. Sup. (3)}, 80:349--426, 1963.

\bibitem[EP08]{EhlersPorter}
Philip~John Ehlers and Timothy Porter.
\newblock Ordinal subdivision and special pasting in quasicategories.
\newblock {\em Adv.\ Math.}, 217(2):489 -- 518, 2008.

\bibitem[FGK{\etalchar{+}}19]{FGKUW}
M.~{Feller}, R.~{Garner}, J.~{Kock}, M.~{Underhill-Proulx}, and M.~{Weber}.
\newblock Every $2$-{S}egal space is unital.
\newblock \href{https://arxiv.org/abs/1905.09580}{arXiv:1905.09580 }, 2019.

\bibitem[FPP08]{FiorePaoliPronk}
Thomas~M. Fiore, Simona Paoli, and Dorette Pronk.
\newblock Model structures on the category of small double categories.
\newblock {\em Algebr. Geom. Topol.}, 8(4):1855--1959, 2008.

\bibitem[GCKT18]{GalvezKockTonks}
Imma G{\'a}lvez-Carrillo, Joachim Kock, and Andrew Tonks.
\newblock Decomposition spaces, incidence algebras and {M}\"obius inversion
  {I}: {B}asic theory.
\newblock {\em Adv. Math.}, 331:952--1015, 2018.

\bibitem[GP99]{GP}
Marco Grandis and Robert Pare.
\newblock Limits in double categories.
\newblock {\em Cahiers Topologie G\'eom. Diff\'erentielle Cat\'eg.},
  40(3):162--220, 1999.

\bibitem[Hau18]{HaugsengIteratedSpans}
Rune Haugseng.
\newblock Iterated spans and classical topological field theories.
\newblock {\em Math. Z.}, 289(3-4):1427--1488, 2018.

\bibitem[Hir03]{Hirschhorn}
Philip~S. Hirschhorn.
\newblock {\em Model categories and their localizations}, volume~99 of {\em
  Mathematical Surveys and Monographs}.
\newblock American Mathematical Society, Providence, RI, 2003.

\bibitem[Hol08]{Hollander}
Sharon Hollander.
\newblock A homotopy theory for stacks.
\newblock {\em Israel J. Math.}, 163:93--124, 2008.

\bibitem[Joy02]{joyalquasicategories}
Andr{\'e} Joyal.
\newblock Quasi-categories and {K}an complexes.
\newblock {\em J.\ Pure Appl.\ Algebra}, 175(1-3):207--222, 2002.

\bibitem[Joy08a]{Joyal}
Andr{\'e} Joyal.
\newblock Notes on quasi-categories.
\newblock preprint available at
  \url{https://www.math.uchicago.edu/~may/IMA/Joyal.pdf}, 2008.

\bibitem[Joy08b]{JoyalVolumeII}
Andr{\'e} Joyal.
\newblock The theory of quasi-categories and its applications.
\newblock preprint available at
  \url{http://mat.uab.cat/~kock/crm/hocat/advanced-course/Quadern45-2.pdf},
  2008.

\bibitem[JT07]{JoyalTierney}
Andr\'e Joyal and Myles Tierney.
\newblock Quasi-categories vs {S}egal spaces.
\newblock In {\em Categories in Algebra, Geometry and Mathematical physics},
  volume 431 of {\em Contemp. Math.}, pages 277--326. Amer. Math. Soc.,
  Providence, RI, 2007.

\bibitem[Kel82]{Kelly}
Max Kelly.
\newblock {\em Basic Concepts of Enriched Category Theory}, volume~64.
\newblock Cambridge University Press, 1982.

\bibitem[Kel90]{KellerCh}
Bernhard Keller.
\newblock Chain complexes and stable categories.
\newblock {\em Manuscripta Math.}, 67(4):379--417, 1990.

\bibitem[Len18]{lenz}
Tobias Lenz.
\newblock Homotopy (pre)derivators of cofibration categories and
  quasicategories.
\newblock {\em Algebr. Geom. Topol.}, 18(6):3601--3646, 2018.

\bibitem[Lur09]{htt}
Jacob Lurie.
\newblock {\em Higher Topos Theory}, volume 170 of {\em Annals of Mathematics
  Studies}.
\newblock Princeton University Press, Princeton, NJ, 2009.

\bibitem[Lur17]{LurieHA}
Jacob Lurie.
\newblock Higher algebra.
\newblock available at \url{http://www.math.harvard.edu/~lurie/papers/HA.pdf},
  2017.

\bibitem[{Pen}17]{PenneyHall}
Mark~D. {Penney}.
\newblock {The universal Hall bialgebra of a double 2-Segal space}.
\newblock {\em arXiv:1711.10194}, 2017.

\bibitem[{Pog}17]{Poguntke}
Thomas {Poguntke}.
\newblock {Higher Segal structures in algebraic $K$-theory}.
\newblock {\em arXiv:1709.06510}, 2017.

\bibitem[Qui67]{QuillenHA}
Daniel~G. Quillen.
\newblock {\em Homotopical {A}lgebra}.
\newblock Lecture Notes in Mathematics, No. 43. Springer-Verlag, Berlin-New
  York, 1967.

\bibitem[Qui73]{QuillenK}
Daniel Quillen.
\newblock Higher algebraic {$K$}-theory. {I}.
\newblock In {\em Algebraic {$K$}-theory, {I}: {H}igher {$K$}-theories ({P}roc.
  {C}onf., {B}attelle {M}emorial {I}nst., {S}eattle, {W}ash., 1972)}, pages
  85--147. Lecture Notes in Math., Vol. 341. Springer, Berlin, 1973.

\bibitem[RV15]{RiehlVerity2Cat}
Emily Riehl and Dominic Verity.
\newblock The 2-category theory of quasi-categories.
\newblock {\em Advances in Mathematics}, 280:549 -- 642, 2015.

\bibitem[SS03]{ss}
Stefan Schwede and Brooke Shipley.
\newblock Stable model categories are categories of modules.
\newblock {\em Topology}, 42(1):103--153, 2003.

\bibitem[Wal85]{waldhausen}
Friedhelm Waldhausen.
\newblock Algebraic {$K$}-theory of spaces.
\newblock In {\em Algebraic and Geometric Topology ({N}ew {B}runswick,
  {N}.{J}., 1983)}, volume 1126 of {\em Lecture Notes in Math.}, pages
  318--419. Springer, Berlin, 1985.

\bibitem[{Wal}16]{Walde}
Tashi {Walde}.
\newblock {Hall monoidal categories and categorical modules}.
\newblock {\em arXiv 1611.08241}, 2016.

\bibitem[You18]{Young}
Matthew~B. Young.
\newblock Relative 2-{S}egal spaces.
\newblock {\em Algebr. Geom. Topol.}, 18(2):975--1039, 2018.

\end{thebibliography}

\end{document}